\documentclass[11pt]{amsart}

\usepackage{graphicx}
\usepackage{amssymb}
\usepackage{latexsym}
\usepackage{epstopdf}
\allowdisplaybreaks[1]  

\newtheorem{thm}{Theorem}[section]

\newtheorem{prop}[thm]{Proposition}
\newtheorem{defn}{Definition}

\renewcommand{\t}{\tau}
\newcommand{\s}{\sigma}
\newcommand{\me}{\widetilde{\mu_\e}}
\newcommand{\mei}{\widetilde{\mu_{\e_j}}}
\newcommand{\lei}{\mathcal{L}_{\e_j}}
\newcommand{\tr}{\mathrm{Tr}}
\newcommand{\dvol}{\,d\mathrm{vol}}

\renewcommand{\div}{\mathrm{div}}
\newcommand{\grad}{\mathrm{grad}}
\newcommand{\D}{\mathcal{D}}
\newcommand{\Q}{\mathcal{Q}}
\newcommand{\Ne}{\mathcal{N}_\varepsilon}
\renewcommand{\L}{\Lambda}
\newcommand{\e}{\varepsilon}
\newcommand{\R}{\mathbb{R}}
\newcommand{\se}{{^{\e\hspace{-.03in}}\sigma}}
\newcommand{\so}{{^{_{0}\hspace{-.025in}}\sigma}}
\newcommand{\Oo}{{^{0\hspace{-.01in}}\Omega}}
\newcommand{\Oe}{{^{\e\hspace{-.01in}}\Omega}}
\newcommand{\cle}{\lesssim}
\newcommand{\A}{\mathcal{A}}
\newcommand{\G}{\mathcal{G}}
\newcommand{\C}{\mathcal{C}}
\renewcommand{\L}{\mathcal{L}}
\renewcommand{\H}{\mathcal{H}}
\newcommand{\N}{\mathbb{N}}
\newcommand{\sM}{{^{\ _{\hspace{-.04in}M}\hspace{-.03in}}\sigma}}
\newcommand{\pM}{i_\e}
\newcommand{\pMi}{i_{\e_j}}
\newcommand{\gM}{{^{\ _{\hspace{-.04in}M}\hspace{-.03in}}g}}
\newcommand{\go}{{^{\ _{\hspace{-.04in}0}\hspace{-.03in}}g}}
\newcommand{\pe}{\iota_\e}
\newcommand{\pei}{\iota_{\e_j}}
\newcommand{\matho}{$\ddot{o}$}
\renewcommand{\o}{$\ddot{\mathrm{o}}$}
\newcommand{\wef}{\widehat{w_{\e,\Phi}}}
\newcommand{\dnabla}{\nabla^{^{^{\hspace {-.08in}\delta}}}\hspace{.01in}}
\newcommand{\gpnabla}{\nabla^{^{^{\hspace {-.12in}g_P}}}\hspace{-.02in}}
\newcommand{\genabla}{\nabla^{^{^{\hspace {-.12in}g_\e}}}\hspace{.01in}}
\newcommand{\gnabla}{\nabla^{^{^{\hspace {-.07in}g}}}\hspace{.01in}}
\newcommand{\gonabla}{\nabla^{^{^{\hspace {-.12in}g_0}}}\hspace{.002in}}
\newcommand{\X}{\mathbf{X}}
\def\crn#1#2{{\vcenter{\vbox{
        \hbox{\kern#2pt \vrule width.#2pt height#1pt
           }
          \hrule height.#2pt}}}}
\newcommand{\intprod}{\mathchoice\crn54\crn54\crn{3.75}3\crn{2.5}2}
\newcommand{\into}{\mathbin{\intprod}}

	\numberwithin{equation}{section} 

\title{A Gluing Construction Regarding Point Particles in General Relativity}
\author{Iva Stavrov Allen}


\address{Department of Mathematical Sciences, Lewis \& Clark College, Portland, OR 97219, USA. Phone: 1-503-768-7559,\ Fax: 1-503-768-7668.}
\email{istavrov@lclark.edu}

\begin{document}
\begin{abstract}
We develop a gluing construction which adds scaled and truncated asymptotically Euclidean solutions of the Einstein constraint equations to compact solutions with potentially non-trivial cosmological constants. The result is a one-parameter family of initial data which has ordinary and scaled ``point-particle" limits analogous to those of Gralla and Wald \cite{sb}. In particular, we produce examples of initial data which generalize Schwarzschild - de Sitter initial data and  gluing theorems of IMP-type \cite{IMP}.
\end{abstract}
\maketitle

\section{The subject of this paper}
\subsection{The motivation}
Understanding the motion of a body of finite size is generally a difficult problem; classically a simple description of the motion is only possible if one makes a ``point-particle" approximation. ``Point-particles", however, are not possible in general relativity because, informally speaking, too much mass in too little space creates ``black holes". In providing a justification for the MiSaTaQuWa equations (which are believed to govern the motions of a small body in general relativity), Gralla and Wald \cite{sb} interpret the ``point particle limit" in terms of a one parameter family of space-times  $\gamma_\e$ which satisfies various limit conditions as $\e\to 0$. One intuitively thinks of $\e\to 0$ as a ``body shrinking to zero size", the meaning of which we make precise below. The existence of interesting one-parameter families of space-times satisfying the conditions of \cite{sb} is, at the time of writing, not well understood. In this paper we construct  examples of initial data which have, among other features, the potential to produce families of space-times with limit properties of \cite{sb}. 

\subsection{The background on initial data} General relativists model space-time as a four dimensional Lorentzian manifold which satisfies the Einstein field equations. If the space-time is globally hyperbolic, that is, if there exists a spacelike hypersurface $M$ such that every inextendible timelike curve intersects $M$, then the Einstein field equations allow an initial value formulation. The initial data for the problem are the induced Riemannian metric $g$ on our Cauchy slice $M$, and the inherited second fundamental form $K$. The Gauss-Codazzi equations together with the Einstein field equations impose restrictions on $(M,g,K)$, which are known as \emph{the Einstein constraint equations} or simply: \emph{the constraints}. The constraint coming from the Codazzi equation is often referred to as \emph{the momentum constraint}, while the Gauss's equation corresponds to \emph{the Hamiltonian constraint}.  

For simplicity we consider only the case where no matter fields are present, although we allow a non-zero cosmological constant $\Lambda$. 
In this situation the constraints are:
$$\begin{aligned}
&\text{Momentum Constraint:}\ \ \div K-\nabla (\tr_g K)=0\\
&\text{Hamiltonian Constraint:}\ R(g)-|K|^2_g+(\tr_g K)^2-2\Lambda=0;
\end{aligned}$$
the term $R(g)$ refers to the scalar curvature of $(M,g)$. For fixed $M$ and $\Lambda$, solving the constraints for a Riemannian metric $g$  and a symmetric two-tensor $K$ yields initial data which can be evolved to produce space-times \cite{Yvonne}. We work in the smooth category and leave the lower regularity to be examined in subsequent works. 

There are a handful of basic types of initial data. One type, which we refer to as \emph{large-scale} or \emph{cosmological} initial data is used to model the  ``universe as a whole". While there are no significant restrictions on the topology for these solutions to the constraints, in this paper we only consider \emph{compact} solutions. This assumption on the topology simplifies the analysis and the function spaces we use, but there are no reasons to believe that it is absolutely necessary. Another feature of these large-scale solutions is the possibility of a non-zero cosmological constant $\Lambda$. As an example consider de Sitter initial data. 

A second type of data is used to model isolate bodies. One example is data corresponding to the Schwarzschild space-time. It is customary in this context to set $\Lambda=0$, and require that the underlying geometry be \emph{asyptotically Euclidean} (AE) in the following sense. 

\begin{defn}\label{AE}
By \emph{AE initial data} we mean a solution of the constraints $(M_0, g_0, K_0)$ for which 
there exists a compact subset $D\subseteq M_0$ and a diffeomorphism 
$$\so:\bigsqcup \left(\R^3\smallsetminus \bar B_{C}\right)\to \left(M_0\smallsetminus D\right)$$ having the following properties: 
\begin{enumerate}
\item The map $\so$ is a diffeomorphism between a disjoint union of complements of closed balls $\bar B_C=\{x\in\R^3\ |\ |x|\le C\}$ and the set $M_0\smallsetminus D$.
\medbreak
\item There exists a constant $c>0$ such that the canonical components of the pullback data satisfy point-wise estimates
$$\Big|(\so^{*}g_0)_{ij}-\delta_{ij}\Big|\le \frac{c}{|x|} \text{\ \ and\ \ }
\Big|(\so^{*}K_0)_{ij}\Big|\le \frac{c}{|x|^2}.$$
\item For every multiindex $\beta$ with $|\beta|\ge 1$ we have a constant $c_\beta$ and  point-wise estimates
$$\Big|\partial^\beta(\so^{*}g_0)_{ij}\Big|\le\frac{c_\beta}{|x|^{|\beta|+1}}, \ \ \ \Big|\partial^\beta(\so^{*}K_0)_{ij}\Big|\le\frac{c_\beta}{|x|^{|\beta|+2}}.$$
\end{enumerate} 
\end{defn}
Thus the geometry of AE data is progressively more and more ``flat" as we move deeper into an asymptotic end, i.e. as we move away from the ``center of mass of the isolated body". 
In this paper 
we refer to $M_0\smallsetminus D$ as \emph{asymptotia}, and each connected component of $M_0\smallsetminus D$ as an \emph{asymptotic end}, of the AE data. 
We note that each asymptotic end corresponds to a ``point at infinity" in the following sense. If an AE data $(M_0, g_0, K_0)$ has $n$ asymptotic ends, then there is a compact manifold $\overline{M}_0$ and points $I_1, I_2, ..., I_n\in \overline{M}_0$ such that $M_0$ is diffeomorphic to $\overline{M}_0\smallsetminus\{I_1, I_2, ..., I_n\}$ under a diffeomorphism which maps the asymptotic ends to deleted open neighborhoods of $I_1$, ..., $I_n$. One may view $\overline{M}_0$ as a compactification of $M_0$. In this paper we refer to  points $I_1, ..., I_n\in \overline{M}_0$ as the \emph{asymptotic endpoints of $M_0$}.

\subsection{The main example}\label{mainexample} 
The motivating example of \cite{sb} is a one-parameter family of Schwarzschild - de Sitter space-times. We revisit the example here in terms of the corresponding initial data. Following \cite{sb} we suppress the topology of the data and study the situation on a single coordinate patch only. 

Both the Schwarzschild and de Sitter space-times admit a Cauchy slice on which   the second fundamental form vanishes and the metric can  be expressed as
$$ds^2=\left(1-F(r)\right)^{-1}dr^2+r^2d\omega^2$$
in some spherical coordinates. The constraints for such data reduce to the Hamiltonian constraint
$R(g)=2\Lambda$ which, after expressing the scalar curvature in terms of $F$, becomes a linear ODE:
\begin{equation}\label{babyH}
F+rF'=\Lambda r^2.
\end{equation}
A homogeneous solution of  this ODE is $F_0(r)=\tfrac{2M}{r}$ ($M$ fixed), a function which corresponds to the Schwarzschild initial data.
It is customary to think of  the parameter $M>0$ as the mass of the isolated body whose center of mass is at $r=0$ and the event horizon is at $r=2M$.  \emph{Scaling} $F_0$ by a factor of $\e$ shrinks the mass of the body down to $M\e$ and  ``brings the asymptotia and the event horizon closer in". Thus, one may interpret the family of functions $\e F_0$, $\e\to 0$, as modeling a body which ``shrinks to zero size". 

One particular solution of \eqref{babyH} is $F(r)=\tfrac{\Lambda}{3}r^2$, a function corresponding to the de Sitter data. Adding the ``shrinking" Schwarzschild body to the de Sitter background amounts to  considering the general solution of \eqref{babyH}:  
$$F_\e(r)=F(r)+\e F_0(r)=\tfrac{\Lambda}{3}r^2+\tfrac{2M\e}{r}.$$
The solution $F_\e$ corresponds to the one-parameter family of metrics $g_\e$ for Schwarzschild - de Sitter data: 
\begin{equation}\label{Schw-deSitter}
ds^2=\left(1-\tfrac{\Lambda}{3}r^2-\tfrac{2M\e}{r}\right)^{-1}dr^2+r^2d\omega^2.
\end{equation}
Let us proceed by examining what happens if we take $\e\to 0$, that is, if we take  a ``point-particle limit"  of $g_\e$. First of all, we note that the $2M\e$-term in \eqref{Schw-deSitter} converges to zero (away from $r=0$) and that $g_\e$ approach the de Sitter metric $ds^2=\left(1-\tfrac{\Lambda}{3}r^2\right)^{-1}dr^2+r^2d\omega^2$. This is to be expected as the contribution of a ``point-particle" to a large-scale geometry ought to be negligible. What is perhaps more interesting is ``zooming in" and paying attention to small scale behavior near $r=0$. To that end we consider the metric $\tfrac{1}{\e^2}g_\e$ in scaled coordinates $(\rho,\omega)$ where $\rho=\tfrac{r}{\e}$:  
$$ds^2=\left(1-\e^2\tfrac{\Lambda}{3}\rho^2-\tfrac{2M}{\rho}\right)^{-1}d\rho^2+\rho^2d\omega^2.$$
In the limit as $\e\to 0$ we recover the original Schwarzschild body!
 
\subsection{The two point-particle limit properties} Let us now precisely state   the two limit properties illustrated above. The terminology we use is motivated by that of \cite{sb}.

Let $(M,g,K)$ be large-scale initial data, let $S\in M$ and let $(M_0,g_0, K_0)$ be AE initial data.  In the example above $(M,g,K)$ corresponds to de Sitter data with $S$ as the origin in $\R^3$, and $(M_0, g_0, K_0)$  the Schwarzschild data. A family $(M_\e, g_\e, K_\e)$ of initial data (Schwarzschild-de Sitter in the example) obeys point-particle limit properties with respect to $(M,g,K)$, $(M_0,g_0,K_0)$ and $S\in M$ if the following hold.

\begin{enumerate}
\item  {\bf The ordinary point-particle limit property.} Let $\mathbf{K}\subseteq M\smallsetminus\{S\}$ be a compact set. For small enough $\e$ there exist embeddings $i_\e:\mathbf{K}\to M_\e$ such that for all $k\in\mathbb{N}\cup\{0\}$ 
$$\|(i_\e)^*g_\e-g\|_{C^{k}(\mathbf{K},g)}\to 0,\ \ \|(i_\e)^*K_\e-K\|_{C^{k}(\mathbf{K},g)}\to 0 \text{\ \ as\ \ }\e\to 0.$$

\medbreak
\item {\bf The scaled point-particle limit property.} Let $\mathbf{K}\subseteq M_0$ be a compact set. 
For small enough $\e$ there exist embeddings $\iota_\e:\mathbf{K}\to M_\e$ such that for all $k\in\mathbb{N}\cup\{0\}$
$$\left\|\tfrac{1}{\e^2}\left(\iota_\e\right)^*g_\e-g_0\right\|_{C^k(\mathbf{K},g_0)}\to 0,\
\ \left\|\tfrac{1}{\e}\left(\iota_\e \right)^*K_\e- K_0\right\|_{C^k(\mathbf{K},g_0)}\to 0 \text{\ \ as\ \ }\e\to 0.$$ 
\end{enumerate}

From this point on every use of the word ``shrinking" should be interpreted in terms of the scaled point-particle limit property. 

\subsection{The results} Our main result is a gluing construction which adds scaled and truncated AE initial data to compact cosmological initial data. The outcome of the construction is a one-parameter family of initial data which has ordinary and scaled point-particle limits as described above. Our analytical work is based upon the IMP-gluing procedure \cite{IMP} and, in particular, the conformal method. We work with data $(M,g, K)$ and $(M_0, g_0, K_0)$ for which $\tau:=\tr_g K$ and $\tr_{g_0}K_0$ are constant. The advantage of considering such constant mean curvature (CMC) data is that the conformal method becomes particularly user-friendly. 

The IMP-gluing techniques typically require certain non-degeneracy conditions (e.g. \cite{engineering}, \cite{IMP}) which  are generically satisfied \cite{KIDS}.  In our case we need to assume the following conditions on the large-scale initial data:
\begin{description}
\item[{\sc CKVF assumption}] $(M,g)$ has no conformal Killing vector fields;
\item[{\sc Injectivity assumption}] The operator $\Delta_g-\left(|K|^2_g-\Lambda\right)$ on $(M,g)$ has trivial kernel. 
\end{description} 
We address these assumptions in more detail in an appendix to this paper. 
For now note that if the cosmological constant $\Lambda$ is relatively small, that is $\Lambda<|K|^2_g$,  the {\sc Injectivity Assumption} is unnecessary due to   the Maximum Principle. 

The following is our main result.

\begin{thm}\label{MAINTHM}
Let $(M,g,K)$ be a (not necessarily connected) compact and smooth CMC solution of the vacuum constraints corresponding to a cosmological constant $\Lambda$ such that  the {\sc CKVF} and the {\sc Injectivity assumption} are satisfied. Furthermore, let $(M_0, g_0, K_0)$ be a smooth AE CMC solution of the vacuum constraints with no cosmological constant. Let $n$ be the number of asymptotic ends of $(M_0,g_0)$, let $S_1, S_2,..., S_n\in M$ and let $I_1, I_2, ..., I_n\in\overline{M}_0$ be the asymptotic endpoints of $M_0$. Finally, let $\nu\in\left(\tfrac{3}{2},2\right)$.

There exists $\e_0>0$ such that for each $\e\in\left(0,\e_0\right)$ there is a  
compact and smooth CMC solution 
$(M_\e, g_{\e}, K_{\e})$  of the vacuum constraints corresponding to the cosmological constant $\Lambda$ 
with the following properties.
\begin{enumerate}
\item 
$M_\e$ is diffeomorphic to an $n$-fold connected sum of $M$ and $\overline{M}_0$ obtained by excising small balls around $S_1, ..., S_n$  and $I_1, ..., I_n$ and identifying the boundaries of the balls at $S_1$ and $I_1$, $S_2$ and $I_2$,..., $S_n$ and $I_n$.
\medbreak
\item For $\mathbf{K}\subseteq M\smallsetminus\{S_1, S_2, ...., S_n\}$ a compact set and $\e$ small enough, depending on $\mathbf{K}$, there exist embeddings $i_\e:\mathbf{K}\to M_\e$ such that for all $k\in\mathbb{N}\cup\{0\}$ we have
$$\|(i_\e)^*g_\e- g\|_{C^{k}(\mathbf{K},g)}=O(\e^{\nu/2}),\ \ \  
\|(i_\e)^*K_\e- K\|_{C^{k}(\mathbf{K},g)}=O(\e^{\nu/2}) \text{\ \ \ as\ \ }\e\to 0.$$
\medbreak
\item For $\mathbf{K}\subseteq M_0$ a compact set and $\e$ small enough, depending on $\mathbf{K}$, there exist embeddings $\iota_\e:\mathbf{K}\to M_\e$ such that for all $k\in\mathbb{N}\cup\{0\}$ we have
$$\|\tfrac{1}{\e^2}(\iota_\e)^*g_\e- g_0\|_{C^{k}(\mathbf{K},g_0)}=O(\e^{1-\nu/2}),\  
\|\tfrac{1}{\e}(\iota_\e)^*K_\e- K_0\|_{C^{k}(\mathbf{K},g_0)}=O(\e^{1-\nu/2}) \text{\ \ \ as\ \ }\e\to 0.$$
\end{enumerate}
\end{thm}

Our result provides examples of initial data which are generalizations of Schwarzschild - de Sitter initial data. In future work we intend to investigate if one could use these examples to 
generate one-parameter families of space-times required in \cite{sb}. 

Another consequence of our main theorem is that one can use \emph{any} AE initial data as a ``shrinking prefabricated bridge" connecting several compact initial data so long as the compact data satisfy {\sc CKVF} and  {\sc Injectivity assumptions}. This generalizes IMP-type results  \cite{IMP} in which the geometry of the gluing region is, essentially, a shrinking Cauchy slice in the Kruskal extension of Schwarzschild space-time. The idea behind this application of our result is related to the work of D. Joyce in \cite{Joyce}.

\subsection{The gluing method} Here we review the steps of a typical initial data gluing procedure. 
\begin{itemize}
\item {\bf The topological gluing.} For a connected sum of compact initial data one excises balls of radius ``$\e$"  and  identifies the surrounding small annular regions, $\mathcal{A}_{1,\e}$ and $\mathcal{A}_{2,\e}$, using  inversion\footnote{IMP-construction does not literally involve inversion. Small neighborhoods of two points, $S_1$ and $S_2$ are first conformally blown up to become asymptotic cylinders, which are then truncated and glued to each other along what used to be the geodesic polar coordinate $r$. Up to an additive constant IMP-procedure identifies the logarithm of the geodesic distance from $S_1$, i.e. $\ln(r_1)$, with the negative of the logarithm of the geodesic distance from $S_2$, $-\ln(r_2)$. In effect, $r_1$ is identified with (a small constant multiple of) $\tfrac{1}{r_2}$. This  is nothing but inversion.}. 
The main difference between our work and the gluing procedures in the literature thus far is that we do not use inversion but directly identify the nearly Euclidean shrinked asymptotia (truncated to an annular shape) with an ``almost" Euclidean annular region near the center of a geodesic normal coordinate patch in the large-scale data. 
\medbreak
\item {\bf The approximate initial data.} One uses cut-offs with differentials supported in the annular regions $\mathcal{A}_{1,\e}$, $\mathcal{A}_{2,\e}$ to combine the Riemannian structures and the second fundamental forms. The derivatives of the cut-offs make the resulting ``data"  $(M_\e, g_\e, K_\e)$ violate the constraints in such a way that the ``size" of the violation approaches zero  as $\e\to 0$. We note that $(M_\e, g_\e, K_\e)$ can at least be made CMC by killing off the trace-free part $\mu:=K-\frac{\tau}{3}g$ in the region where the metric $g_\e$ transitions from $g\big|_{\mathcal{A}_{1,\e}}$ to $g\big|_{\mathcal{A}_{2,\e}}$. 
\medbreak
\item {\bf Repairing the momentum constraint.} The idea here is to perturb $K_\e$ or, rather, its trace-free part $\mu_\e:=K_\e-\tfrac{\tau}{3}g_\e$ so that the momentum constraint  is satisfied. This is done using the conformal Killing operator $\D_\e$ and the vector Laplacian $L_\e=\D_\e^*\D_\e$ on $(M_\e,g_\e)$;  
the perturbation $\me$  is given by 
$\me:=\mu_\e+\D_{\e} X_\e$, where $L_\e X_\e=(\div_{g_\e}\mu_\e)^{\sharp}$. 
\medbreak
\item {\bf The Lichnerowicz equation.} The  Hamiltonian constraint can be  repaired using  conformal changes which  do not break the momentum constraint we just fixed; see \cite{bartnik-isenberg} for details.  The conformal change we make in our paper is of the form  $g_\e\to \phi_\e^4 g_\e$, $\mu_\e+\tfrac{\tau}{3}g_\e\to \phi_\e^{-2}\me+\tfrac{\t}{3}\phi_\e^4g_\e$, where $\phi_\e$ is a positive solution of the \emph{Lichnerowicz equation}
\begin{equation}\label{LICHN}
\Delta_{g_\e}\phi_\e - \tfrac{1}{8}R(g_\e)\phi_\e+\tfrac{1}{8}|\me|_{g_\e}^2\phi_\e^{-7}+\left(\tfrac{\Lambda}{4}-\tfrac{\t^2}{12}\right)\phi_\e^5=0.
\end{equation}
\end{itemize}
By making careful estimates in geometrically-adapted \emph{weighted H\"older spaces} one shows that the resulting one-parameter family of glued data obeys any desired limit properties.

\subsection{The organization of the paper}
The presentation of our gluing procedure begins with {\sc  A List of Notational Conventions}. The details of the topological gluing and the creation of the approximate data can be found in the following section, 
{\sc The Approximate Data}. 

We proceed in the section 
 {\sc Weighted Function Spaces} to give a detailed description of several weighted H\"older spaces used in our construction. Not only do these spaces permit a unified approach to showing that \emph{both} point-particle limit properties hold, they also provide a context in which many relevant facts
(e.g.~Theorem \ref{ell-M;S} and its consequence \eqref{VL:est2})
 can be most easily demonstrated and understood. It is our opinion that their importance warrants a clear, albeit somewhat lengthy, exposition.

The remainder of the paper is dedicated to the analysis of the PDE's needed to repair the momentum and the Hamiltonian constraints. This material is organized in sections {\sc Repairing the Momentum Constraint} and {\sc The Lichnerowicz Equation}. Each of these two sections ends with an examination of the point-particle limit properties. 

Our paper has two appendices, the first of which is {\sc On the CKVF and  Injectivity assumptions}. The second appendix concerns (the kernel of) the Euclidean vector Laplacian $L_\delta$ which is involved in repairing the momentum constraint. We have found two ways of executing this step of the proof, one being the direct computation of the kernel. The approach included in the main body of our paper is more elegant and  more in the spirit of our other arguments. The downside of the indirect approach is that it only works for a narrower set of weights (which turns out to be necessary elsewhere in the paper). In this respect the direct computation of $\mathrm{ker}(L_\delta)$ is more optimal. Having found no resource in the literature with an explicit treatment of $L_\delta$ and its kernel we have decided to append our computation of $L_\delta$ to this article. Thus, our paper ends with {\sc Appendix B. The Euclidean Vector Laplacian in Spherical Coordinates}.

\subsection{Acknowledgments} I use this opportunity to thank Paul T.~ Allen and James Isenberg for encouragement during my transition into this area of mathematical general relativity. The majority of the work for this particular project was done during a Junior Sabbatical Leave from Lewis and Clark College in 2008, while visiting the Geometric Analysis Group of The Albert Einstein Institute in Golm, Germany. I am particularly grateful to Prof.~ Gerhard Huiskin for his generous support of my visit. I also thank V. Moncrief for helpful discussions regarding the material in {\sc Appendix A.} 

\section{A List of Notational Conventions}
\begin{itemize}
\item We let $f(\e,x)\cle g(\e,x)$
mean that there is some constant $C>0$ such that for all points $x$ on a specified domain, and all sufficiently small $\e>0$ we have $f(\e, x)\le C g(\e, x)$. In particular, the symbol $\cle$ is to be interpreted as being uniform in $\e$. We use $f(\e, x)\sim g(\e, x)$ as an abbreviation for $g(\e, x)\cle f(\e, x)\cle g(\e, x)$. In case of additional parameters we let  $$f(\e, x, \Phi)\cle g(\e, x, \Phi) \text{\ \ {\it independently of}\ \ } \Phi$$ mean that the constant $C>0$ can be chosen independently of $\Phi$; the same comment applies to $\sim$. 
\medbreak
\item We use letters $k,\alpha, \nu$ in the context of weighted function spaces. Unless specifically stated otherwise, the reader should assume that 
$k\in \N\cup\{0\}$, $\alpha\in[0,1)$ and $\nu\in \R$. 
\medbreak
\item We use $p,q\in \N\cup\{0\}$ in denoting the type of a tensor. Unless specifically stated otherwise, a tensor $T$ should be assumed to be of type $T^{i_1...i_p}_{j_1...j_q}$. If the tensor type is clear from context $p,q$ may not be emphasized. 
\medbreak
\item We use $\delta$ to denote the Euclidean metric. 
\medbreak
\item In principle, $B_r$ (resp.~ $\bar{B}_r$) denotes an open (resp.~ closed) ball of radius $r$ in an ambient space which should be clear from the context. Unless explicitly stated otherwise, the reader should assume that the ball is centered at the origin. 
\medbreak
\item We use $|\cdot|$ to denote the point-wise norm of tensors. This symbol is typically decorated by an index which reveals which metric has been used to compute the norm. If the symbol $|\cdot|$ is left undecorated the reader should assume that we are discussing the Euclidean norm or the absolute value. 
\medbreak
\item The symbols for geometric differential operators are decorated by a superscript or a subscript which indicates the metric with respect to which the operators are defined. For example, $\gpnabla$ denotes the covariant differentiation with respect to some metric $g_P$, while $\Delta_\delta$ denotes the Euclidean scalar Laplacian. 
\end{itemize}

\section{The Approximate Data}\label{glue:sec}
Let $(M,g,K)$ be a smooth compact CMC solution of the vacuum constraints with cosmological constant $\Lambda$. Assume that $(M,g,K)$ satisfy {\sc CKVF} and {\sc Injectivity assumptions}. Let 
$\tau:=\tr_g K$ and $\mu:=K-\frac{\tau}{3}g$. The 2-tensor $\mu$ is trace-free and, by virtue of the momentum constraint, also divergence-free.

Let $(M_0, g_0, K_0)$ be a smooth asymptotically Euclidean  CMC vacuum initial data with no cosmological constant; adopt the notation established in Definition \ref{AE}. \emph{For simplicity we outline the gluing procedure assuming $(M_0,g_0)$ has one asymptotic end}.  For sufficiently large $R$ we set $\Oo_{R}:=\so\big(\R^3\smallsetminus \bar{B}_R\big)$. The AE-estimates of Definition \ref{AE} imply $\tr_{g_0} K_0=0$; let $\mu_0:=K_0$. 

\subsection{Scaling}   
Consider $(M_0,\e^2g_0,\e K_0)$; these data also satisfy the constraints. 
By precomposing $\so$ with 
$x\mapsto \tfrac{x}{\e}$
one obtains ``scaled" coordinates $\se$ for the asymptotia of $M_0$:
\begin{equation*}
\se:\big(\R^3\smallsetminus \bar{B}_{C\e}\big)\to M_0.
\end{equation*}
We let $\rho(X):=|\se^{-1}(X)|=\frac{|{\so^{-1}}(X)|}{\e}$ denote our new radial function in the asymptotia. 
Note the following AE-estimates: 
\begin{align}
\label{asymptg}
&\Big|(\se^{*}\e^2g_0)_{ij}-\delta_{ij}\Big|\cle \frac{\e}{\rho} 
\mathrm{\ \ \ and\ \ \ }
\Big|\partial^\beta(\se^{*}\e^2g_0)_{ij}\Big|\cle\frac{1}{(\frac{\rho}{\e})^{|\beta|+1}}\cdot\frac{1}{\e^{|\beta|}}\cle\frac{\e}{\rho^{|\beta|+1}}
\mathrm{\ \ \ if\ \ \ }|\beta|\ge 1,\\
\label{asymptK}
&\Big|\partial^\beta(\se^{*}\e K_0)_{ij}\Big|\cle \frac{1}{\e}\cdot\frac{1}{(\frac{\rho}{\e})^{|\beta|+2}}\cdot\frac{1}{\e^{|\beta|}}
\cle\frac{\e}{\rho^{|\beta|+2}} \mathrm{\ \ if\ \ } |\beta|\ge 0.
\end{align}
By truncating our scaled AE data we mean considering  
$M_0\smallsetminus  \Oe _{C^{-1}}$
where 
$$ \Oe _{C^{-1}}:=\ \se\Big(\big\{x\in \R^3\big|\ |x|>C^{-1}\big\}\Big)=\Oo_{(\e C)^{-1}}.$$

\subsection{The topological gluing}\label{gluetop:sec}
Consider $S\in M$ and geodesic normal coordinates $\sM$ with respect to $(M,g)$ centered at $S$. By increasing the value of $C$ if necessary we may assume that these coordinates are defined on  $B_{C^{-1}}$. We have 
\begin{equation}\label{metricM}
\big(\sM^* g\big)_{ij}=\delta_{ij}+l_{ij}(x)
\end{equation}
where $l_{ij}(x)=O(|x|^2)$ as $|x|\to 0$. Moreover, the first derivatives of $l_{ij}$ satisfy $\partial l_{ij}(x)=O(|x|)$ while the higher derivatives satisfy  $\partial^\beta l_{ij}(x)=O(1)$  as $|x|\to 0$. 
We proceed by excising a small geodesic ball $B^M_{C\e}:=\sM\left(B_{C\e}\right)$.
\begin{defn}
Let $\sim_\e$ be the smallest equivalence relation on $\widehat{M_\e}:=\Big(M\smallsetminus B^{M}_{C\e}\Big)\cup \Big(M_0\smallsetminus  \Oe _{C^{-1}}\Big)$ for which $P\sim_\e Q$ whenever $\sM^{-1}(P),\ \se^{-1}(Q)$ are both defined and satisfy $\sM^{-1}(P)\ =\ \se^{-1}(Q)$.
We define the family of manifolds $M_\e$ by
$M_\e:=\widehat{M_\e}\big/_{{\sim}_\e}.$
\end{defn}
There are natural quotient maps
\begin{equation*}
\pM:\Big(M\smallsetminus B^M_{C\e}\Big)\to M_\e \mathrm{\ \ and\ \ }
\pe:\Big(M_0\smallsetminus  \Oe _{C^{-1}}\Big) \to M_{\e}
\end{equation*} 
which we may use to put coordinates on $M_\e$.  Note that 
$\sigma:=\pM\circ\sM=\pe\circ\se$
maps the annular region $\A_\e:=\{x\in \R^3\ |\ C\e\le |x| \le C^{-1}\}$ to the \emph{gluing region}  
\begin{equation*}
\Big[\se\big(\A_\e\big)\Big]\Big/_{\sim_\e}=\Big[\sM(\A_\e)\Big]\Big/_{\sim_\e}\subseteq M_\e.
\end{equation*}
The map $\s$ defines our preferred coordinates for the gluing region. For the purposes of dealing with point-particle limits note that for 
any compact subset $\mathbf{K}\subseteq M\smallsetminus\{S\}$ and sufficiently small  $\e$ the quotient map $\pM$ gives rise to an embedding of $\mathbf{K}$ into $M_\e$. Likewise, if $\mathbf{K}\subseteq M_0$ is a compact subset then for $\e$ small enough the quotient map $\pe$ 
defines an embedding of $\mathbf{K}$ into $M_\e$. 

\subsection{Combining the metrics} We use partition of unity to combine the metrics $g$ and $\e^2g_0$. Let $\chi$ be a decreasing smooth cut-off function on $\R$ such that 
\begin{equation*}
\chi\equiv 1 \mathrm{\ \ on\ \ } (-\infty,1] \mathrm{\ \ \ \ and\ \ \ \ }\chi\equiv 0\mathrm{\ \ on\ \ } [4,+\infty).
\end{equation*} 
We define $g_\e$ on $M_\e$ to be the unique metric for which  $\pM^* g_\e=g$ and $\pe^*g_\e=\e^2g_0$ \emph{away} from the gluing region while in the gluing region itself we have
\begin{equation}\label{defn-g}
\sigma^*g_\e\Big|_x=\chi\Big(\tfrac{|x|}{\sqrt{\e}}\Big)\big[\se^*\e^2g_0\big]\Big|_x+
			(1-\chi)\Big(\tfrac{|x|}{\sqrt{\e}}\Big)\big[\sM^*g\big]\Big|_x.
\end{equation}
Loosely speaking, our metric $g_\e$ ``matches" with $g$ on the part of the gluing region given by $|x|\ge 4\sqrt{\e}$ and it ``matches" with $\e^2g_0$ on $|x|\le \sqrt{\e}$. 

\subsection {Combining the second fundamental forms}\label{mu:glue} The cut-off $\chi$ is also involved in combining the trace-free parts, $\mu$ and $\e\mu_0$, of $K$ and $\e K_0$.
We define $\mu_\e$ to be the unique symmetric 2-tensor for which $\pM^*\mu_\e=\mu$ and $\pe^*\mu_\e=\e \mu_0$ \emph{away} from the gluing region while in the gluing region itself we have 
\begin{equation}\label{defn-mu}
\sigma^*\mu_\e\Big|_x=\chi\Big(\tfrac{6}{\sqrt{\e}}|x|-2\Big)\big[\se^*\e\mu_0\big]\Big|_x+
				(1-\chi)\Big(\tfrac{3}{4\sqrt{\e}}|x|-2\Big)\big[\sM^*\mu\big]\Big|_x.
\end{equation}
It follows from (\ref{defn-mu}) that  on the part of the gluing region given by
\begin{itemize}
\item $\sqrt{\e}\le |x|\le 4\sqrt{\e}$ we have $\mu_\e\equiv 0$;
\medbreak
\item $|x|\ge 4\sqrt{\e}$ we have that $\mu_\e$ is a multiple of (the pullback of) $\mu$; on $|x|\ge 8\sqrt{\e}$ the tensor $\mu_\e$ exactly ``matches" $\mu$;
\medbreak
\item $|x|\le \sqrt{\e}$ we have that $\mu_\e$ is a multiple of (the pullback of) $\e\mu_0$; on $|x|\le\tfrac{1}{2}\sqrt{\e}$ the tensor $\mu_\e$ exactly ``matches" $\e\mu_0$.
\end{itemize}
Since $\mu$ and $\mu_0$ are traceless with respect to $g$ and $g_0$ respectively,  
we see that 
$\tr_{g_\e}\mu_\e =0$.
It possible that $\div_{g_\e}\mu_\e\neq 0$, with its support in the part of the gluing region given by $\tfrac{1}{2}\sqrt{\e}\le |x|\le 8\sqrt{\e}$. In other words, the data $\left(M_\e, g_\e, \mu_\e+\tfrac{1}{3}\t g_\e\right)$ where 
$\t=\tr_g K$
is only an \emph{approximate solution} of the constraints.

\section{Weighted Function Spaces}
Our analysis relies on weighted H\o lder spaces described below. A reader may find it useful to examine the geometric relationship between the four types of spaces. 
\begin{enumerate}
\item Weighted H\o lder spaces on $(M_\e,g_\e)$ which in some sense measure the (point-wise) size of tensors relative to the proximity to the gluing region; 
\medbreak
\item Weighted H\o lder spaces on $(M\smallsetminus \{S\},g)$ which keep track of (point-wise) decay and/or growth of tensors in terms of the geodesic distance from $S$; 
\medbreak
\item Weighted H\o lder spaces on $(M_0,g_0)$ which keep track of the decay of tensors with respect to  the radial function of the ``asymptotia".
\medbreak
\item Weighted H\"older spaces on the Euclidean $\left(\R^3\smallsetminus\{0\},\delta\right)$ which keep track of the decay and/or growth of tensors near $0$ and $\infty$ in terms of the radial function $r:x\mapsto |x|$.
\end{enumerate}
All four of these function spaces are defined using preferred atlases for the underlying manifolds. The charts in these atlases are in turn determined by the respective weight functions. The weight functions are in no way ``canonical" - we choose them so that they are conducive to obtaining the desired results. We only need the weight functions to obey properties such as those stated in our Proposition \ref{prop-covers}. This  approach to weighted H\"older spaces has been highly influenced by \cite{CD} (see the appendix on ``scaling properties") and 
\cite{lee}, and we thank the authors of these two papers for addressing the function spaces in detail. Since the weighted spaces on $(M_\e, g_\e)$ are the most delicate of all we present their construction in detail; the remaining weighted spaces are discussed only briefly.

\subsection{The weight function for $(M_\e, g_\e)$}\label{weight-sec} 
Consider smooth, increasing, positive functions $w_{0,\R}$ and $w_{M,\R}$ on $\R$ for which 
$$w_{0,\R}(s)= \begin{cases}
			s &\mathrm{if}\ \ s\ge 2C\\
			C &\mathrm{if}\ \ s\le \frac{3}{2}C,
			\end{cases}
	\mathrm{\ \ and\ \ }
w_{M,\R}(s)= \begin{cases}
			C^{-1} &\mathrm{if}\ \ s\ge (\frac{3}{2}C)^{-1}\\
			s &\mathrm{if}\ \ s\le (2C)^{-1}.
			\end{cases}		
			$$
We define the weight function $w_\e$ to be the unique function on $M_\e$ which satisfies 
$$\Big(w_\e\circ\s\Big) (x)=
\begin{cases}
w_{M,\R}(|x|) &\mathrm{\ \ if\ \ } |x|\ge 2C\e\\
\e\cdot w_{0,\R}\big(\tfrac{|x|}{\e}\big) & \mathrm{\ \ if\ \ } |x|\le (2C)^{-1} 
\end{cases}
$$
for $x\in \A_\e$ and which is constant away from $\s(\A_\e)$. In other words, 
\begin{itemize}
\item $w_\e\equiv C^{-1}$ on the component of $M_\e\smallsetminus \s(\A_\e)$ arising from $M$;
\item  $w_\e$ on $\s(\A_\e)$ in some sense describes the distance in $(M,g)$ away from $S$. Equivalently, one may think of $w_\e$ on $\s(\A_\e)$ as being the radial function in the asymptotia of $(M_0,\e^2g_0)$;
\item $w_\e\equiv C\e$ on the connected component of 
$M_\e\smallsetminus \s(\A_\e)$ arising from $M_0$.
\end{itemize}

\subsection{Special atlases for $M_\e$}\label{atlases-sec} For each $M_\e$ we  construct two atlases, one a refinement of the other. Within these atlases we distinguish three types of charts, depending on which region of $M_\e$ they cover. The charts are described in detail below, but we take the opportunity to outline the three types first. One type ({\sc Type $\G$}) are the charts whose images are well within the gluing region. To be more precise, these charts cover 
$\G_\e\subseteq\s(\A_\e)$  where $\G_\e$ is determined by 
$$\s^{-1}(\G_\e)=\left\{x\in \R^3\ \big|\ 4C\e<|x|<(4C)^{-1}\right\}.$$ There are two connected components of $M_\e\smallsetminus \G_\e$: one arising from $M$ and one arising from $M_0$. By {\sc Type $M$} we mean the charts (described below) which cover the component arising from $M$ and by {\sc Type $M_0$} we mean the charts which cover the component arising from $M_0$. 

\begin{description}
\item[{\sc Type $\G$}] To a point $P=\s(x_P)\in\G_\e$ we associate the chart $\s_P=\s\circ \H_P$ where 
$$\H_P:B_1\to B\Big(x_P; \tfrac{|x_P|}{2}\Big),\ \ \ \ \H_P(x)=x_P+\tfrac{|x_P|}{2}\cdot x.$$
This chart is well defined since 
$$C\e<2C\e<\tfrac{|x_P|}{2}<\big|\H_P(x)\big|< \tfrac{3}{2}|x_P|<\tfrac{3}{8}C^{-1}<C^{-1} \mathrm{\ \ \ for\ all\ \ \ }x\in B_1.$$
It is important to note that   
$w_\e\big(\s_P(x)\big)=|\H_P(x)|$ for all $x\in B_1$.
To a chart $\s_P$ we associate the scaled pullback 
\begin{equation}\label{defn-gP}
g_P:=\tfrac{4}{|x_P|^2}\s_{P}^{\ *}g_\e.
\end{equation}
Finally, we also consider the restriction $\s_P'$ of $\s_P$ to  $B_{\frac{1}{2}}$.
\medbreak
\item[{\sc Type $M$}]  Consider  finitely many charts $\sM_1, ..., \sM_N$ which cover $M\smallsetminus B_{(4C)^{-1}}^M$ whose domains are $B_1$  and whose images are contained in 
$M\smallsetminus B_{(8C)^{-1}}^M$. Moreover, assume that their restrictions to $B_\frac{1}{4}$ cover $M\smallsetminus B_{(4C)^{-1}}^M$.  By composing $\sM_1, ..., \sM_N$ with $\pM$ we obtain charts
$$\s_n:=\pM\circ\sM_n, \ \ \ 1\le n\le N$$
which 
cover the component of $M_\e\smallsetminus\G_\e$ arising from $M$. 
To $\s_n$ we associate the pullback  $$g_n:=\s_n^{\ *}g_\e=(\sM_n)^{*}g.$$ 
The restrictions of $\s_n$ to $B_\frac{1}{2}$ will be denoted by $\s_n'.$ 
\medbreak
\item[{\sc Type $M_0$}] 
Consider finitely many\footnote{Note that, without loss of generality, we may assume that the number $N$ of charts of {\sc Type $M$} is the same as the number of charts of {\sc Type $M_0$}.} charts $\so_1, ... ,\so_N$ covering $M_0\smallsetminus \Oo_{4C}=M_0\smallsetminus  \Oe _{4C\e}$ whose domains are $B_1$, whose images are contained in $M_0\smallsetminus \Oo _{8C}$, and whose restrictions to balls $B_\frac{1}{4}$ cover $M_0\smallsetminus  \Oo _{4C}$. By composing with $\pe$ we obtain 
$$\s_{-n}:=\pe\circ\so_n,\ \ \ 1\le n\le N,$$
charts which cover the component of $M_\e\smallsetminus\G_\e$ arising from $M_0$. To a chart $\s_{-n}$ we associate the scaled pullback 
$$g_{-n}:=\tfrac{1}{\e^2}\  \s_{-n}^{\ *}\ g_\e =(\so_n)^{*}g_0.$$
Note that $g_{-n}$ are independent of $\e$. 
The restrictions of $\s_{-n}$ to $B_\frac{1}{2}$ will be denoted by $\s_{-n}'$.
\end{description}
The collection of charts $\s_P$ for $P\in\G_\e$ and $\s_n,\ \s_{-n}$ for $1\le n\le N$ will be denoted by $\C_\e$.  The collection of their restrictions to $B_\frac{1}{2}$ will be denoted by $\C'_\e$.

We list some properties of $\C_\e$ and $\C_\e'$; the reader may find it useful to compare these with the scaling properties of \cite{CD}.
\begin{prop}\label{prop-covers}\hfill
\begin{enumerate}
\item We have $w_\e\big(\Phi(x)\big)\sim w_\e\big(\Phi(0)\big)$ independently of $\Phi\in\C_\e$. 
\item The metrics $g_P$ for $P\in\G_\e$ are uniformly close to $\delta$ in the Euclidean  $C^k$-norm.
\item The transition functions for the charts in $\C_\e$ and $\C_\e'$ have uniformly bounded Euclidean H\matho lder norms. In particular,
\begin{equation}\label{expan.factor}
\frac{|F(x)-F(y)|}{|x-y|}\sim 1
\end{equation}
independently of the transition function $F$.
\item We have
$\big|\dnabla^k(w_\e\circ\Phi)^\nu\big| \cle \big(w_\e\circ\Phi)^\nu$
independently of $\Phi\in\C_\e$.
\end{enumerate}
\end{prop}

\begin{proof} It is immediate that for all $x\in B_1$ and all charts $\Phi$ of {\sc Type  $\G_\e$} we have 
$$\tfrac{1}{2}w_\e\big(\Phi(0)\big)\le w_\e\big(\Phi(x)\big)\le \tfrac{3}{2} w_\e\big(\Phi(0)\big).$$ For the remaining charts the claim (1) follows from the fact that $w_\e\sim 1$ on $\pM\Big(M\smallsetminus B_{(8C)^{-1}}\Big)$ and  $w_\e\sim \e$ on $\pe\Big(M_0\smallsetminus  \Oe _{8C\e}\Big)$. 

To prove part (2) of our proposition 
let $P\in\G_\e$ and $x\in B_1$. It follows from the definition (\ref{defn-gP}) of $g_P$, the definition (\ref{defn-g}) of $g_\e$, the estimate (\ref{asymptg}), and expansion (\ref{metricM}) that the components of $g_P$ satisfy 
\begin{equation}\label{unif-eucl}
\begin{aligned}
\big| (g_P)_{ij} -\delta_{ij}\big|\ \Big|_x=&\big|(\sigma^*g_\e)_{ij}-\delta_{ij}\big|\ \Big|_{\H_P(x)}\\
\le&\big|(\se^*\ \e^2g_0)_{ij}-\delta_{ij}\big|\ \Big|_{\H_P(x)}+
			\big|(\sM^*\ g)_{ij}-\delta_{ij}\big|\ \Big|_{\H_P(x)}\cle \frac{\e}{\big|\H_P(x)\big|}+\big|\H_P(x)\big|^2.
\end{aligned}
\end{equation}
Since $\e\cle \big|\H_P(x)\big|\cle 1$ independently of $P\in \G_\e$ we see that $\big| (g_P)_{ij} -\delta_{ij}\big|$ is uniformly bounded from above. In fact, the same argument can be applied to the derivatives of $(g_P)_{ij}$.
The only complication here is the presence of the derivatives of the cut-off function $\chi\Big(\frac{|\H_P(x)|}{\sqrt{\e}}\Big).$ 
However, each derivative of this function is uniformly bounded on $B_1$. To see this, decompose
$$\chi\Big(\frac{|\H_P(x)|}{\sqrt{\e}}\Big)=\widehat{\chi}\circ \widehat{\H_P}(x),\ \mathrm{\ for\ }\ \widehat{\chi}(y)= \chi(|y|) \mathrm{\ \ and\ \ } \widehat{\H_P}(x)=\tfrac{1}{\sqrt{\e}}\cdot x_P+\tfrac{|x_P|}{2\sqrt{\e}}\cdot x.$$
All derivatives of $\widehat{\chi}$ are bounded on $\R^3$, and are supported in the annulus $1\le |y|\le 4$. Since 
\begin{equation}\label{deriv-cut-off}
\partial^\beta\big(\widehat{\chi}\circ \widehat{\H_P}\big)=\big(\tfrac{|x_P|}{2\sqrt{\e}}\big)^{|\beta|}\cdot\big[(\partial^\beta\widehat{\chi})\circ \widehat{\H_P}\big]
\end{equation}
for all multi-indices $\beta$, we see that $\partial^\beta\big(\widehat{\chi}\circ \widehat{\H_P}\big)$ is supported on the set $\{x\in B_1\ \big|\ 1\le \widehat{\H_P}(x)\le 4\}$. 
Moreover, $\tfrac{|x_P|}{2\sqrt{\e}}\le \widehat{\H_P}(x)\le \tfrac{3|x_P|}{2\sqrt{\e}}$ for all $x\in B_1$ and so 
$\partial^\beta\big(\widehat{\chi}\circ \widehat{\H_P}\big)\neq 0$ only if 
$$\tfrac{1}{3}\le \tfrac{|x_P|}{2\sqrt{\e}}\le 4.$$
It now follows from (\ref{deriv-cut-off}) that $\partial^{\beta}_x\Big[\chi\big(\tfrac{|\H_P(x)|}{\sqrt{\e}}\big)\Big]$ for each individual multi-index $\beta$ is bounded independently of $\e$, $P$ and $x$.  Arguing as in (\ref{unif-eucl}) completes  the proof of part (2) of our proposition.

To prove part (3) it suffices to consider the transition functions between two charts one of which  is of {\sc Type $\G$}. A transition function $\s_{P}^{-1}\circ\s_{Q}$ for $P, Q\in\G_\e$ is a composition of a translation and a dilation with coefficient $\frac{w_\e(Q)}{w_\e(P)}$. In order for $\mathrm{Im}(\s_P) \cap \mathrm{Im} (\s_Q)\neq \emptyset$ we need to have  $\frac{w_\e(Q)}{w_\e(P)}\in[\frac{1}{3}, 3]$. Therefore, all the derivatives of the transition functions of the form $\s_{P}^{-1}\circ \s_{Q}$ are bounded from above by $3$. Inspecting the condition $\mathrm{Im}(\s_n)\cap \mathrm{Im}(\s_P)\neq \emptyset$ we see that the dilation coefficient $\Big(\frac{w_\e(P)}{2}\Big)^{\pm 1}$ of $\H_{P}^{\pm 1}$ is bounded uniformly in $\e$. Since there are only finitely many  transition functions  $\s_{n}^{-1}\circ \s$ and $\s^{-1}\circ\s_n$ we see that $\s_{n}^{-1}\circ \s_P$ and $\s_{P}^{-1}\circ \s_{n}$ have all their derivatives bounded uniformly in $\e$. Similar considerations prove the boundedness in the last remaining cases: those of $\s_{-n}^{-1}\circ \s_P$ and $\s_{P}^{-1}\circ \s_{-n}$. Roughly speaking, here we have $w_\e(P)=O(\e)$ while $\s=\pe\circ\se$ features a scaling by $\e$ such that  the resulting net contribution of $\s_P$ and $\s_{P}^{-1}$ in terms of $\e$ is simply $O(1)$.  The explicit details are left to the reader. 

Part (4) of the proposition is immediate for the charts of {\sc Type $M$} due to their finite number. Since $w_\e\circ\pe(X)=\e\,w_{0,\R}\left(\left|\so^{-1}(X)\right|\right)$ for all  $X\in\Oo_C$, the same is true for charts of {\sc Type $M_0$}. Thus, it remains to understand the charts of {\sc Type $\G$}. 

In general, we have an estimate of the form
\begin{equation*}
\Big|\dnabla^k\Big((w_\e\circ\Phi)^\nu\Big)\Big|\le \sum_{j=1}^k\left[c_j(w_\e\circ\Phi)^\nu\sum\Big|\frac{\dnabla^{i_1}(w_\e\circ\Phi)}{w_\e\circ\Phi}\Big|\otimes...
\otimes\Big|\frac{\dnabla^{i_j}(w_\e\circ\Phi)}{w_\e\circ\Phi}\Big|\right].
\end{equation*}
Therefore, it suffices to prove that  
\begin{equation}\label{weight<=}
\Big|\dnabla^k \big(w_\e\circ \Phi\big)\Big| \cle (w_\e\circ \Phi)
\end{equation}
independently of $\Phi=\s_P$, $P=\s(x_P)\in\G_\e$.
For $k=1$ one can directly compute:
\begin{equation*}
\Big|\frac{d\big(w_\e\circ\Phi\big)}{w_\e\circ \Phi}(x)\Big|  
=\left|\frac{\sum_{a=1}^3(x_P^a+\frac{|x_P|}{2}x^a)\cdot \frac{|x_P|}{2}}{|x_P+\frac{|x_P|}{2}x|^2}dx^a\right| \le 
\frac{\frac{|x_P|}{2}}{|x_P+\frac{|x_P|}{2}x|}\cdot \sum_{a=1}^3\frac{|x_P^a+\frac{|x_P|}{2}x^a|}{|x_P+\frac{|x_P|}{2}x|}\le 3.
\end{equation*}
The remainder of the proof of \eqref{weight<=} is easily done by induction. To avoid notational complications we only do the case of 
$k=2$ and leave the general induction step to the reader. 
We have
\begin{equation*}
\dnabla^2\big((w_\e\circ \Phi)^2\big)=2d\big(w_\e\circ \Phi)\otimes d\big(w_\e\circ\Phi)+2(w_\e\circ\Phi)\dnabla^2(w_\e\circ\Phi),
\end{equation*}
while 
$\dnabla^2\big((w_\e\circ\Phi)^2\big)(x)=\dnabla^2\big(\sum_{a=1}^3 (x_P^a+\frac{|x_P|}{2}x^a)^2\big)$, viewed as a matrix, is a product of $\big(\frac{|x_P|}{2}\big)^2$ and the identity. Dividing  by $(w_\e\circ \Phi)^2$, using the boundedness of $\frac{\big(\frac{|x_P|}{2}\big)^2}{(w_\e\circ\Phi)^2}$ and the already established estimate for $\Big|\frac{d(w_\e\circ\Phi)}{w_\e\circ\Phi}\Big|$ we obtain 
\begin{equation*}
\Big|\frac{\dnabla^2(w_\e\circ\Phi)}{w_\e\circ\Phi}\Big|\le \Big|\frac{\dnabla^2\big((w_\e\circ\Phi)^2\big)}{(w_\e\circ\Phi)^2}\Big|+\Big|\frac{d\big(w_\e\circ\Phi\big)}{w_\e\circ\Phi}\Big|^2\cle 1.\qedhere
\end{equation*} 

%
\end{proof}

\subsection{H$\ddot{\mathbf{o}}$lder norms on $M_\e$} We start by  defining the norms. The reader may find it instructive to compare our approach to that of \cite{lee}.
\begin{defn}\label{holder-def}
If $T$ is a tensor field with locally $C^{k,\alpha}$ components, 
then set $$\|T\|_{k,\alpha}:=\sup_{\Phi\in\C_\e}\|\Phi^*T\|_{C^{k,\alpha}(B_1,\delta)} \ \ \ \text{and}\ \ \ 
\|T\|_{k,\alpha}':=\sup_{\Phi\in\C_\e'}\|\Phi^*T\|_{C^{k,\alpha}(B_\frac{1}{2},\delta)}.$$
\end{defn}

The norms defined above are referred to as ``H\o lder norms" or ``$C^{k,\alpha}$-norms". To develop an intuition about them we work out an equivalent form of the $C^{k,0}$-norm. 

We start by analyzing $\|(\s_P)^*T\|_{C^{k,0}}$ for a chart $\s_P$ of {\sc Type $\G$} and a tensor $T=T^{i_1i_2...i_p}_{j_1j_2...j_q}$ on $M_\e$. We have
\begin{equation*}
\big|\gpnabla^k (\s_{P}^* T)\big|_{g_P}=\big|\gpnabla^k (\s_{P}^* T)\big|_{\frac{4}{w_\e(P)^2} \s_{P}^*g_\e}=\Big(\frac{w_\e(P)}{2}\Big)^{-p+q+k}\cdot \Big(\Big|\genabla^k T\Big|_{g_\e}\circ \s_P\Big).
\end{equation*}
Proposition \ref{prop-covers} implies that $w_\e\circ \s_P\sim w_\e(P)$ independently of $P$ and that 
for each $k, p, q$ there exists $C_{p,q,k}>0$ such that 
\begin{equation*}
\begin{aligned}
w_{\e}^{-p+q+k}\big|\genabla^k T\big|_{g_\e}\circ \s_{P} &\le C_{p,q,k}\Big(\big|\dnabla^k \s_{P}^*T\big| +...+\big|\dnabla \s_{P}^*T\big|  +\big|\s_{P}^*T\big| \Big)\\
\big|\dnabla^k \s_{P}^*T\big| &\le C_{p,q,k}\Big(w_{\e}^{-p+q+k}\big|\genabla^k T\big|_{g_P}+...+w_{\e}^{-p+q+1}\big|\genabla  T\big|_{g_\e}+w_{\e}^{-p+q}\big|T\big|_{g_\e}\Big)\circ \s_P,
\end{aligned}
\end{equation*}
point-wise.  Analogues of these estimates also hold for charts of {\sc Type $M$} and {\sc Type $M_0$}. Indeed, one may apply the reasoning outlined above to metrics $g_{\pm n}$ and   take advantage of the fact that $w_\e\circ \s_n\sim 1$, while $w_\e\circ \s_{-n}\sim \e$. It follows that the $\|\cdot\|_{k,0}$-norm on the space of tensors $T=T^{i_1...i_p}_{j_1...j_q}$ is equivalent to the norm 
$$\sum_{j=0}^k \sup_{M_\e}\big[w_\e^{-p+q+j} \big|\genabla^jT\big|_{g_\e}\big],$$
with equivalence constants depending on $p, q, k$ but \emph{not} on $\e$. As the same argument applies to the $\|\cdot\|'_{k,0}$-norm we have the uniform equivalence 
$\|\cdot\|_{k,0}\sim \|\cdot\|'_{k,0}$. 
In fact, more is true:
\begin{prop} \label{holder-equiv}
The norms $\|\cdot\|_{k,\alpha}$ and $\|\cdot\|'_{k,\alpha}$ are equivalent uniformly in $\e$. 
\end{prop}
\begin{proof}
Since $\|T\|'_{k,\alpha}\le \|T\|_{k,\alpha}$ it remains to show that for all $\Phi\in\C_\e$ and all tensors $T$ of a given type we can bound
$$\frac{\ \Big|\dnabla^k \Phi^*T(x)-\dnabla^k \Phi^*T(y)\Big|  \ }{|x-y|  ^\alpha}$$
from above by a uniform multiple of $\|\tilde{\Phi}^*T\|_{C^{k,\alpha}(B_{\frac{1}{2}}, \delta)}$ for some $\tilde\Phi\in\C_\e'$; the word ``uniform" here should be interpreted to mean ``independent of $\e$, $\Phi$, $T$, $x,y$, and $\tilde{\Phi}$". To this end, let $m$ 
be a uniform upper bound implied by (\ref{expan.factor}). 
For a given $\Phi\in\C_\e$, tensor $T$ and $x,y\in B_1$ let
$$x=x_0, x_1, x_2, ... , x_{8m-1}, x_{8m}=y$$ be the division points of the line segment $xy$ into $8m$ congruent subsegments. Note that $|x-y|  <2$ implies $|x_a-x_{a+1}|  <\frac{1}{4m}$ for all $0\le a\le 8m-1$. Since by assumption the restrictions of elements of $\C_\e$ to $B_{\frac{1}{4}}$ cover $M_\e$ we know that for each $a$ there is $\Phi_a$ such that 
$\Phi(x_a)\in \Phi_a\Big(B_\frac{1}{4}\Big)$.
Observe that 
$$\Phi(x_a), \Phi(x_{a+1})\in\Phi_a\Big(B_\frac{1}{2}\Big), \ 0\le a\le 8m-1.$$  
Applying the triangle inequality we see that 
$$
\frac{\left|\dnabla^k\Phi^*T(x)-\dnabla^k\Phi^*T(y)\right|  }{|x-y|^\alpha }
\le
(8m)^{1-\alpha}\frac{\left|\dnabla^k\Phi^*T(x_{\bar{a}})-\dnabla^k\Phi^*T(x_{\bar{a}+1})\right|  }{|x_{\bar{a}}-x_{\bar{a}+1}|^\alpha}.
$$
for some $0\le \bar{a}\le 8m-1$. In principle, the components of $\dnabla^k\Phi^*T=\dnabla^k\Big(\big(\Phi_{\bar{a}}^{-1}\circ \Phi\big)^*\Phi_{\bar a}^*T\Big)$ at some point $\xi$ can be written as sums of products of the components of the covariant derivatives of $\Phi_{\bar{a}}^*T$ at $\Phi_{\bar{a}}^{-1}\circ \Phi(\xi)$, the derivatives of $\big(\Phi_{\bar{a}}^{-1}\circ \Phi\big)^{-1}$ at $\Phi_{\bar{a}}^{-1}\circ \Phi(\xi)$ and the derivatives of $\Phi_{\bar{a}}^{-1}\circ \Phi$ at $\xi$. The uniform bounds on the derivatives of transition functions and expressions discussed in (\ref{expan.factor}) allow us to control the $C^{0,\alpha}(B_\frac{1}{2},\delta)$-norms of $\partial^\beta\Big(\Phi^{-1}\circ \Phi_{\bar{a}}\Big)\circ \big(\Phi_{\bar{a}}^{-1}\circ \Phi\big)$ and $\partial^\beta\Big(\Phi_{\bar{a}}^{-1}\circ\Phi\Big)$, which then leaves us with 
$$\frac{\Big|\dnabla^k\Phi^*T(x_{\bar{a}})-\dnabla^k\Phi^*T(x_{\bar{a}+1})\Big|  }{|x_{\bar{a}}-x_{\bar{a}+1}|^\alpha  }\cle \|\Phi_{\bar{a}}^*T\|_{C^{k,\alpha}(B_\frac{1}{2},\delta)}$$
independently of $\bar a$, $\Phi$ and $T$. This completes our proof. 
\end{proof}

There is an important feature of the previous proof: it shows that the atlases $\C_\e$ and $\C_\e'$  in Definition \ref{holder-def} can be replaced  by any  \emph{finite} sub-atlases whose charts after restriction to $B_{\frac{1}{4}}$ still cover $M_\e$. Indeed, the only effect such a change of atlases has is that $\|\cdot\|_{k,\alpha}$ (and $\|\cdot\|'_{k,\alpha}$) are being replaced by equivalent norms. Consequently, compactness of $M_\e$ implies that any tensor field $T$ on $M_\e$ with locally $C^{k,\alpha}$ components satisfies 
$$\|T\|_{k,\alpha}<\infty \text{\ \ \ \big(and\ equivalently,\ \ }
\|T\|_{k,\alpha}'<\infty\big).$$

A natural question to investigate at this point is whether the space of all tensors $T^{i_1...i_p}_{j_1...j_q}$ on $M_\e$ with locally $C^{k,\alpha}$ components is complete with respect to the H\o lder norm. A purely  formal exercise which solely uses Definition \ref{holder-def} and completeness of $C^{k,\alpha}(B_1)$ gives a positive answer to the question. 

\begin{defn} 
The H\matho lder norm $\|\cdot\|_{k,\alpha}$ gives the set of all tensors $T^{i_1...i_p}_{j_1...j_q}$ on $M_\e$ with locally $C^{k,\alpha}$ components the structure of a Banach space,
which we denote by $C^{k,\alpha}_{p,q}(M_\e)$; this space is also referred to as a H\matho lder space. For simplicity we often write $C^{k,\alpha}(M_\e)$ in place of  $C^{k,\alpha}_{p,q}(M_\e)$.
\end{defn}

\subsection{Weighted H$\ddot{\mathbf{o}}$lder spaces on $M_\e$} 
\begin{defn} 
The Banach space of all tensors $T^{i_1 ... i_p}_{j_1... j_q}$ with locally $C^{k,\alpha}$ components
and norm(s)
$$\|T\|_{k,\alpha,\nu}:=\|w_\e^\nu T\|_{k,\alpha} \mathit{\ \ \ (and/or\ }
\|T\|_{k,\alpha,\nu}':=\|w_\e^\nu T\|_{k,\alpha}'\ )$$
is denoted by $C^{k,\alpha, \nu}_{p,q}(M_\e)$. The norm $\|\cdot\|_{k,\alpha,\nu}$ is referred to as a weighted H\matho lder norm, and $C^{k,\alpha,\nu}_{p,q}(M_\e)$ is referred to as a weighted H\matho lder space. 
\end{defn}
We point out that the two weighted H\o lder norms ($\|\cdot\|$ and $\|\cdot\|'$) are equivalent \emph{uniformly in} $\e$ (Proposition \ref{holder-equiv}). We proceed by discussing some equivalent representations of our weighted H\o lder norm.

For $\Phi\in\C_\e$ (or $\C_\e'$) set $\wef:=w_\e\big(\Phi(0)\big)$; the reader should think of $\wef$ as ``a sample value" of the weight function $w_\e$ on the chart $\Phi$. 
\begin{prop}\label{main.form}
The weighted H\matho lder norms $\|\cdot\|_{k,\alpha,\nu}$ and $\|\cdot\|'_{k,\alpha,\nu}$ are equivalent to the norms
$$T\mapsto \sup_{\Phi\in\C_\e}\wef^\nu\|\Phi^*T\|_{C^{k,\alpha}(B_1,\delta)} \mathit {\ \ \ and\ \ \ }
T\mapsto \sup_{\Phi\in\C_\e'}\wef^\nu\|\Phi^*T\|_{C^{k,\alpha}(B_{\frac{1}{2}},\delta)}.
$$
This equivalence is uniform in $\e$.
\end{prop}
\begin{proof}
We focus on the norms arising from $B_1$; the norms arising from $B_\frac{1}{2}$ can be dealt with analogously. The main ingredient of our proof is showing the uniform estimate 
\begin{equation}\label{basic.wef<=}
\Big\|\Phi^*\big(w_\e^\nu T\big)\Big\|_{C^{k,\alpha}(B_1,\delta)}\cle
\wef^\nu\|\Phi^*T\|_{C^{k,\alpha}(B_1,\delta)},
\end{equation}
independently of $\Phi$. Since  
$\dnabla^j \Big(\Phi^*(w_\e^\nu T)\Big)
=\sum_{i=0}^j\dnabla^i\Big((w_\e\circ\Phi)^\nu\Big)\otimes\dnabla^{j-i}\big(\Phi^*T\big)$ for all $0\le j\le k$, 
Proposition \ref{prop-covers} implies the point-wise estimate 
\begin{equation*}
\begin{aligned}
\Big|\dnabla^j\big(\Phi^*(w_\e^\nu T)\big)\Big|\cle &(w_\e\circ \Phi)^\nu\sum_{i=0}^j\Big|\dnabla^{j-i}\big(\Phi^*T\big)\Big|\\
\cle&(w_\e\circ\Phi)^\nu\|\Phi^*T\|_{C^{k,\alpha}(B_1,\delta)}\cle \wef^\nu \|\Phi^*T\|_{C^{k,\alpha}(B_1,\delta)}.
\end{aligned}
\end{equation*}
The same line of reasoning gives us 
\begin{gather*}
\frac{\Big|\dnabla^k\big(\Phi^*(w_\e^\nu T)\big)(x) -\dnabla^k\big(\Phi^*(w_\e^\nu T)\big) (y)\Big|}{|x-y|^\alpha}\phantom{nsgiushgrghephgphgphgiguughiwugegughuwegisga}\\
\begin{aligned}
\le &\sum_{j=0}^k\frac{\Big|\dnabla^j\big((w_\e\circ\Phi)^\nu\big)(x)-\dnabla^j\big((w_\e\circ\Phi)^\nu\big)(y)\Big|}{|x-y|^\alpha}\cdot\big|\dnabla^{k-j}\big(\Phi^*T\big)\big|(x)\\
&\ +\sum_{j=0}^k\big|\dnabla^{j}\big((w_\e\circ\Phi)^\nu\big)\big|(y)\cdot \frac{\Big|\dnabla^{k-j}\big(\Phi^*T\big)(x)-\dnabla^{k-j}\big(\Phi^*T\big)(y)\Big|}{|x-y|^\alpha}\\
\cle&\|(w_\e\circ\Phi)^\nu\|_{C^{k+1}(B_1,\delta)}\cdot \|\Phi^*T\|_{C^{k,\alpha}(B_1,\delta)}\cle \wef^\nu\|\Phi^*T\|_{C^{k,\alpha}(B_1,\delta)}
\end{aligned}
\end{gather*}
independently of $\Phi$. 
Estimate (\ref{basic.wef<=}) is now immediate. 

Using $-\nu$ in place of $\nu$ in (\ref{basic.wef<=}) we obtain 
$$\|\Phi^*T\|_{C^{k,\alpha}(B_1,\delta)}=\|\Phi^*\big(w_\e^{-\nu}(w_\e^\nu T)\big)\|_{C^{k,\alpha}(B_1,\delta)}\cle \wef^{-\nu}\|\Phi^*(w_\e^\nu T)\|_{C^{k,\alpha}(B_1,\delta)},$$
i.e. $\wef^\nu\|\Phi^*T\|_{C^{k,\alpha}(B_1,\delta)}\cle \|\Phi^*(w_\e^\nu T)\|_{C^{k,\alpha}(B_1,\delta)}$. Overall, we have
$$\wef^\nu\|\Phi^*T\|_{C^{k,\alpha}(B_1,\delta)}\cle \|\Phi^*(w_\e^\nu T)\|_{C^{k,\alpha}(B_1,\delta)}\cle \wef^\nu\|\Phi^*T\|_{C^{k,\alpha}(B_1,\delta)},$$
which yields the claimed equivalence of norms. 
\end{proof}
As a consequence of Proposition \ref{main.form} and the description of the (unweighted) $C^{k,0}$-norm we have that  
the $\|\cdot\|_{k,0,\nu}$-norm on $C^{k,0,\nu}_{p,q}(M_\e)$ is equivalent to the norm 
\begin{equation}\label{Holder:equiv}
\sum_{j=0}^k \sup_{M_\e}\big[w_\e^{-p+q+\nu+j} \big|\genabla ^j T\big|_{g_\e}\big],
\end{equation}
with equivalence constants possibly depending on $p, q, k, \nu$ but \emph{not} on $\e$. 

Proposition \ref{main.form} also implies that for any fixed $\e$ and  \emph{any} $\nu_1, \nu_2\in \R$ the 
H\o lder norms $\|\cdot\|_{k,\alpha,\nu_1}$ and $\|\cdot\|_{k,\alpha,\nu_2}$ are equivalent. However, it also shows that the norms are \emph{not} equivalent uniformly in $\e$. It is this non-uniformity that makes our weighted H\o lder spaces useful. 

\subsection{Restrictions of (weighted) H\"older norms to open subsets of $M_\e$}\label{Holder:restriction}

We now define function spaces $C^{k,\alpha,\nu}_{p,q}(U)$ for open subsets  $U\subseteq M_\e$. These restrictions of (weighted) H\"older spaces are avoidable, but we have decided to keep them in our work as an efficient book-keeping tool.  

\begin{defn}\label{Hrest:defn}
Let $U\subseteq M_\e$ be an open subset and let
$$\C_{\e;U}:=\left\{\Phi\big|_{\Phi^{-1}(U)}\ \Big|\ \Phi\in \C_\e\right\}.$$
For tensor fields $T$ on $U$ with locally $C^{k,\alpha}$ components define
$$\|T\|_{k,\alpha;U}:=
\sup_{\Phi\in\C_{\e;U}}\|\Phi^*T\|_{C^{k,\alpha}\left(\Phi^{-1}(U),\delta\right)} \ \ \ \text{and}\ \ \ 
\|T\|_{k,\alpha,\nu;U}:=\|w_\e^\nu T\|_{k,\alpha;U}.$$
The Banach space of all tensors $T^{i_1 ... i_p}_{j_1... j_q}$ on $U$  for which 
$\|T\|_{k,\alpha,\nu;U}<\infty$ is denoted by $C^{k,\alpha, \nu}_{p,q}(U)$.
\end{defn}

It is important to notice that the proof of Proposition \ref{main.form} carries over to our new set-up and that we have the norm equivalences
\begin{gather}\label{main.restriction}
\sup_{\Phi\in \C_{\e;U}} \wef^\nu \|\Phi^*T\|_{C^{k,\alpha}\left(\Phi^{-1}(U),\delta\right)} \cle \|T\|_{k,\alpha,\nu;U} \cle \sup_{\Phi\in \C_{\e;U}} \wef^\nu \|\Phi^*T\|_{C^{k,\alpha}\left(\Phi^{-1}(U),\delta\right)},\\
\|T\|_{k,0,\nu}\sim \sum_{j=0}^k \sup_{U}\big[w_\e^{-p+q+j+\nu}\big|\genabla^j T|_{g_\e}\big],\label{normeq:restr}
\end{gather}
with equivalence constants which are \emph{independent} of $U\subseteq M_\e$.  A careful reader has surely noticed that we have not defined the norm $\|\cdot\|_{k,\alpha,\nu;U}'$. The reason for this is twofold: the norm $\|\cdot\|_{k,\alpha,\nu}'$ is only needed to obtain elliptic estimates on the entire $M_\e$ (e.g. Theorem \ref{unif-ell}), and, the proof of Proposition \ref{holder-equiv} does not carry over as it highly depends on the ``convexity" of $U$. Given this situation we have decided it is best to avoid $\|\cdot\|_{k,\alpha,\nu;U}'$-norm altogether. 

It is clear from Definition \ref{Hrest:defn} that 
$$U\subseteq V\subseteq M_\e\ \Longrightarrow\ \|T\|_{k,\alpha,\nu;U}\le \|T\|_{k,\alpha,\nu;V}$$
for all choices of $k,\alpha,\nu$ and $T$. Other properties of the restricted H\"older norms which we need are listed below; they follow directly from  \eqref{main.restriction} and their proofs are left to the reader.
\begin{prop}\label{Holder:props}
Let $U\subseteq M_\e$ be an open set.
\begin{enumerate}
\item If $\nu_1<\nu_2$ then 
$$(\inf_U w_\e)^{\nu_2-\nu_1}\|T\|_{k,\alpha,\nu_1;U}\cle \|T\|_{k,\alpha, \nu_2;U}\cle (\sup_U w_\e)^{\nu_2-\nu_1}\|T\|_{k,\alpha, \nu_1;U}$$
and in particular
$$\e^{\nu_2-\nu_1}\|T\|_{k,\alpha,\nu_1;U}\cle \|T\|_{k,\alpha, \nu_2;U}\cle \|T\|_{k,\alpha, \nu_1;U}$$
independently of $U\subseteq M_\e$ and $C^{k,\alpha}$-tensor fields $T$. 
\item The point-wise tensor product satisfies 
$$\|T_1\otimes T_2\|_{k,\alpha, \nu_1+\nu_2;U}\cle \|T_1\|_{k,\alpha,\nu_1:U}\|T_2\|_{k,\alpha, \nu_2;U}$$
independently of $U\subseteq M_\e$ and $C^{k,\alpha}$-tensor fields $T_1, T_2$ of a particular type.
\item Contraction of a tensor field gives rise to a continuous linear map 
$$\mathbf{C}:C^{k,\alpha, \nu}_{p+1,q+1}(U)\to C^{k,\alpha, \nu}_{p,q}(U)$$
whose norm is bounded uniformly in $\e$ and independently of $U\subseteq M_\e$.
\item Raising an index of a tensor field on $(U, g_\e)$ gives rise to a continuous linear map
$$^{\sharp}:C^{k,\alpha, \nu}_{p,q+1}(U)\to C^{k,\alpha, \nu+2}_{p+1, q}(U)$$ 
whose norm is bounded uniformly in $\e$ and independently of $U\subseteq M_\e$.
\end{enumerate}
\end{prop}

\subsection{Differential operators and uniform elliptic estimates on $M_\e$}\label{operators:section}
Here are several  
differential operators we use in our analysis: 
\begin{itemize}
\item The conformal Killing operator $\D_\e$, whose action on vector fields is given by $$\big(\D_\e X\big)_{ab}=\frac{1}{2}\big(X_{a;b}+X_{b;a}\big)-\frac{1}{3}\,\left(\div_{g_\e} X\right)\,(g_\e)_{ab}.$$ 
The operator  $\D_\e:C^{k+1,\alpha, \nu}_{1,0}(M_\e)\to C^{k,\alpha,\nu-2}_{0,2}(M_\e)$ has  uniformly bounded norm (see Proposition \ref{De}). It is easy to see that $\mathrm{Im}(\D_\e)$ is contained in the subspace of trace-free symmetric $2$-tensors and that  the formal adjoint  of $\D_\e$ is
$$\D_\e^*: C^{k+1,\alpha,\nu-2}_{0,2}(M_\e)\to C^{k,\alpha,\nu+2}_{1,0}(M_\e) \text{\ \ with\ \ }\D_\e^*(\omega) = -(\div_{g_\e} \omega)^{\sharp}.$$
\medbreak
\item The vector Laplacian $L_\e$, whose action on vector fields is given by $$L_\e X=\D_\e^*\D_\e X=-\div_{g_\e}\big(\D_\e X\big)^{\sharp}.$$
The operators $L_\e$ are self-adjoint and elliptic. They map   $C^{k+2,\alpha, \nu}_{1,0}(M_\e)$ to $C^{k,\alpha, \nu+2}_{1,0}(M_\e)$ and, as we see in Theorem \ref{unif-ell} below, they allow \emph{uniform} elliptic estimates. 
\medbreak
\item The linearized Lichnrowicz operators $\L_\e$, whose action on  functions can schematically be represented as 
$$\L_\e f=\Delta_{g_\e} f-\tfrac{1}{8}R(g_\e) f+h_\e f,$$
where 
$h_\e:M_\e\to \R$ is a family of functions of uniformly bounded $C^{k,\alpha, 2}$-norms: $\|h_\e\|_{k,\alpha, 2}\cle 1$. By Theorem \ref{unif-ell} below, the family of second order (self-adjoint, elliptic) differential  operators $\L_\e: C^{k+2,\alpha, \nu}_{0,0}(M_\e)\to C^{k,\alpha, \nu+2}_{0,0}(M_\e)$ allows \emph{uniform} elliptic estimates.
\end{itemize}
We now prove the above-mentioned properties of $\D_\e$, $L_\e$ and $\L_\e$. 
\begin{prop}\label{De}
If $X\in C^{k+1,\alpha, \nu}_{1,0}(M_\e)$ then $\D_\e X\in C^{k,\alpha, \nu-2}_{0,2}(M_\e)$ and
$$\|\D_\e X\|_{k,\alpha, \nu-2}\cle \|X\|_{k+1,\alpha,\nu}.$$
\end{prop}
\begin{proof} We see from Proposition \ref{main.form}  that it suffices to show  
\begin{equation*}
\|\Phi^*(\D_\e X)\|_{C^{k,\alpha}(B_1,\delta)}\cle \wef^2\|\Phi^* X\|_{C^{k+1,\alpha}(B_1,\delta)}
\end{equation*}
independently of $\Phi\in \C_\e$.  
Our proof  splits into three cases depending on the type of chart $\Phi\in\C_\e$. In the interest of brevity we focus on the most interesting of the three cases, {\sc Type $\G$}.  Let 
$\Phi=\s_P$ with $P\in \G_\e$. We compute:
\begin{equation*}
\begin{aligned}
\|\s_P^*(\D_\e X)\|_{C^{k,\alpha}(B_1,\delta)}=&\|\D_{\s_{P}^*g_\e} \s_P^*X\|_{C^{k,\alpha}(B_1,\delta)}=\tfrac{w_\e(P)^2}{4}\big\|\D_{g_P} \s_P^* X\big\|_{C^{k,\alpha}(B_1,\delta)}\\
\cle &\widehat{w_{\e,\s_P}}^2\|\gpnabla (\s_P^*X)\|_{C^{k,\alpha}(B_1,\delta)}.
\end{aligned}
\end{equation*}
Recall that by Proposition \ref{prop-covers} the linear operator $\gpnabla-\dnabla$ is bounded uniformly in $P\in\G_\e$ and $\e$, together with all of its derivatives; thus 
\begin{equation*}
\begin{aligned}
\|\s_P^*(\D_\e X)\|_{C^{k,\alpha}(B_1,\delta)}\cle &\widehat{w_{\e,\s_P}}^2 \Big[\|\dnabla (\s_P^* X)\|_{C^{k,\alpha}(B_1, \delta)}+\|\s_P^*X\|_{C^{k,\alpha}(B_1,\delta)}\Big]\\
\cle &\widehat{w_{\e,\s_P}}^2 \|\s_P^*X\|_{C^{k+1,\alpha}(B_1,\delta)}.
\end{aligned}
\end{equation*}
The cases of {\sc Type $M$} and {\sc Type $M_0$} are handled in the same manner, using 
\begin{equation}\label{scal-we}
\widehat{w_{\e, \s_n}}\sim 1,\ \  \widehat{w_{\e, \s_{-n}}} \sim\e,
\end{equation}
and the uniformity of $g_{\pm n}$ in $\e$.
\end{proof} 
\begin{thm}\label{unif-ell}
Let $\alpha\in(0,1)$. If  vector field $X\in C^{0,0,\nu}(M_\e)$ is such that $L_\e X\in C^{k,\alpha,\nu+2}(M_\e)$
 then $X\in C^{k+2,\alpha,\nu}(M_\e)$ with:
$$\|X\|_{k+2,\alpha, \nu}\cle \|L_\e X\|_{k,\alpha, \nu+2} +\|X\|_{0,0,\nu}.$$
Likewise, if  function $f\in C^{0,0,\nu}(M_\e)$ satisfies  
$\L_\e f\in C^{k, \alpha, \nu+2}(M_\e)$ then $f\in C^{k+2,\alpha, \nu}(M_\e)$ with 
$$\|f\|_{k+2,\alpha, \nu}\cle \|\L_\e f\|_{k,\alpha, \nu+2} +\|f\|_{0,0,\nu}.$$
\end{thm}
\begin{proof}
Our argument relies on the scaling properties of the vector Laplacian $L_\e$ and the linearized Licherowicz operator $\L_\e$. The reader may find it instructive to compare the following proof with 
that of Proposition \ref{De}. We analyze the vector Laplacian first; it scales as follows:
\begin{equation}\label{scal-Le}
\begin{aligned}
&L_{\s_P^*g_\e}=L_{\frac{w_\e(P)^2}{4}g_P}=4\widehat{w_{\e,\s_P}}^{-2}L_{g_P} \ \ \text{for\ }P\in\G_\e,\text{\ and} \\
&L_{\s_{-n}^*g_\e}=L_{\e^2g_{-n}}=\e^{-2}L_{g_{-n}}\ \  \text{for\ }1\le n\le N. 
\end{aligned}
\end{equation}

By assumption we have $L_h \Phi^*X \in C^{k,\alpha}(B_1,\delta)$ for all  $\Phi\in \C_\e$ and corresponding metrics 
$$h\in \{g_P|P\in\G_\e\}\cup\{g_{\pm n}|1\le n\le N\}.$$ 
We now apply the basic Schauder interior regularity \cite{GT}
to $L_h$ and open sets $B_\frac{1}{2}\subset B_1$. Recall that the constant in the interior regularity estimates \emph{only} depends on $k$,$\alpha$, the (Euclidean)  distance between $B_\frac{1}{2}$ and $B_1$, a lower bound on the eigenvalues of the principal symbol of $L_h$, and an upper bound on the $C^{k, \alpha}$-norm of the coefficients of $L_h$. Since our metrics $h$ are uniformly close to the Euclidean metric $\delta$ by Proposition \ref{prop-covers}, the constants in the interior regularity estimates can be chosen independently of $h$ and $\e$. We now have  $\Phi^*X\in C^{k+2,\alpha}(B_\frac{1}{2},\delta)$ with 
$$\|\Phi^*X\|_{C^{k+2,\alpha}(B_{\frac{1}{2}},\delta)}\cle \|L_h \Phi^*X\|_{C^{k,\alpha}(B_1,\delta)}+\|\Phi^*X\|_{C^{0,0}(B_1,\delta)}$$
independently of $h$, $\Phi$ and $X$.
The scaling properties (\ref{scal-we}) and (\ref{scal-Le}) further imply 
\begin{equation*}
\wef^\nu \|\Phi^*X\|_{C^{k+2,\alpha}(B_{\frac{1}{2}},\delta)} \cle  \wef^{\nu+2}\|\Phi^*L_\e X\|_{C^{k,\alpha}(B_1,\delta)}+\wef^\nu\|\Phi^*X\|_{C^{0,0}(B_1,\delta)}
\end{equation*}
independently of $h$, $\Phi$ and $X$.
Taking supremum over $\Phi$ yields 
$$\|X\|_{k+2,\alpha,\nu}'\cle \|L_\e X\|_{k,\alpha,\nu+2}+\|X\|_{0,0,\nu}<\infty.$$
The claimed uniform elliptic regularity follows from the (uniform) equivalence of $\|\cdot\|$ and $\|\cdot\|'$ norms (Proposition \ref{holder-equiv}). 

In the case of the linearized Lichnerowicz operator $\L_\e$ we rely on the properties of the differential operators $\L_\Phi$ on $B_1$ defined by 
$$\L_\Phi =\wef^2\Phi^* \L_\e.$$
It is clear from the above that it suffices to show the existence of a uniform lower bound (denoted $\underline{\lambda}$) on the eigenvalues of the principal symbol of operators $\L_\Phi$, and a uniform upper bound (denoted $\overline{\lambda}$) on the $C^{k, \alpha}$-norm of the coefficients of $\L_\Phi$; the word ``uniform" here should be interpreted to mean ``independent of $\Phi\in\C_\e$ and $\e$". We only present the argument in the case of $\Phi=\s_{-n}$ and leave the remaining cases ($\Phi=\s_P$, $\Phi=\s_n$) to an interested reader. Observe that 
\begin{equation*}
\begin{aligned}
\L_{\s_{-n}}=&\widehat{w_{\e,\s_{-n}}}^2 \big[\Delta_{\s_{-n}^*g_\e}-\tfrac{1}{8}R(\s_{-n}^*g_\e)+h_\e\circ\s_{-n}\big]\\
=&\frac{\widehat{w_{\e,\s_{-n}}}^2}{\e^{2}}\Big[\Delta_{g_{-n}}-\tfrac{1}{8}R(g_{-n})\Big]+\widehat{w_{\e,\s_{-n}}}^2 h_\e\circ \s_{-n}.
\end{aligned}
\end{equation*}
The existence of the uniform lower bound $\underline{\lambda}$ and the uniform upper bound $\overline{\lambda}$ follows from the independence of $g_{-n}$ in $\e$,  the scaling property (\ref{scal-we}), and the assumption that $\|h_\e\|_{k,\alpha,2}\cle 1$. 
\end{proof}
\subsection{Weighted function spaces on $(M\smallsetminus\{S\}, g)$} In this paper we also need function spaces which measure decay/growth of tensors in terms of the geodesic distance from $S$. The construction of these spaces is analogous to that used for the spaces $C^{k,\alpha,\nu}(M_\e)$. 

The weight function $w_M:M\smallsetminus\{S\}\to \R$ we utilize here can roughly be described as the geodesic distance from $S$. We define $w_M$ as the unique continuous function which is constant  away from $B^M_{C^{-1}}$ and satisfies
$$\big(w_M\circ\sM\big)(x)=w_{M,\R}(|x|)$$
for the function $w_{M,\R}$ of Section \ref{weight-sec}. Note that $w_M=w_\e\circ \pM$ on $M\smallsetminus B^M_{2C\e}$.  

Next, we introduce the special atlases $\C_S$ and $\C_S'$ for $M\smallsetminus\{S\}$. These contain two types of charts:
\begin{enumerate}
\item {\sc Charts near $S$.} Let $\G=\sM\big(\big\{x\ |\ 0<|x|<\frac{C^{-1}}{4}\big\}\big)\subseteq M\smallsetminus\{S\}$ and let $P=\sM(x_P)\in\G$. To point $P$ we associate the chart $\sM_P:=\sM\circ \H_P$ and the metric $$\gM_P=\tfrac{4}{|x_P|^2}\sM_P^*g,$$  where $\H_P:B_1\to \R^3$ is defined by $\H_P(x)=x_P+\frac{|x_P|}{2}x$ (see Section \ref{atlases-sec}). We also consider the restriction $\sM_P'$ of $\sM_P$ to $B_\frac{1}{2}$. Note that $\pM\circ\sM_P=\s_P$ for $|x_P|\ge 4C\e$ and that $\big(w_M\circ \sM_P\big)(x)=\big|\H_P(x)\big|$ for all $x\in B_1$. 
\medbreak
\item {\sc Charts away from $S$.} To cover $M\smallsetminus \left(\G\cup\{S\}\right)$ we use the charts $\sM_1$,..., $\sM_N$ introduced in the Section \ref{atlases-sec}. In addition, we use the restrictions $\sM_1'$,..., $\sM_N'$ of these charts to the ball $B_\frac{1}{2}$, and the metrics $g_n=\sM_n^*g$, $1\le n\le N$ discussed before. 
\end{enumerate}
Define the atlas $\C_S$ as the collection of charts $\sM_P$ for $P\in \G$ and charts $\sM_n$, $1\le n\le N$; the atlas $\C_S'$ consists of the corresponding restrictions. It is important to notice that Proposition \ref{prop-covers} can easily be modified to yield a result regarding $w_M$ and $\C_S$ (in place of $w_\e$ and $\C_\e$). 

The H\o lder norms in this setting are defined as
\begin{equation*}
\|T\|_{k,\alpha}:=\sup_{\Phi\in\C_S}\|\Phi^* T\|_{C^{k,\alpha}(B_1, \delta)},\ \ \ \ \|T\|_{k,\alpha, \nu}:=\|w_M^\nu T\|_{k,\alpha};
\end{equation*}
analogous expressions yield $\|\cdot\|'$. 
Using scaling properties analogous to Proposition \ref{prop-covers} one can prove the following norm equivalences:
\begin{equation*}
\begin{aligned}
&\|T\|_{k,\alpha, \nu}\sim\sup_{\Phi\in \C_S}\widehat{w_{M,\Phi}}^\nu\|\Phi^* T\|_{C^{k,\alpha}(B_1, \delta)}\sim \|T\|_{k,\alpha,\nu}'\sim \sup_{\Phi\in\C_S'}\widehat{w_{M,\Phi}}^\nu\|\Phi^* T\|_{C^{k,\alpha}(B_\frac{1}{2}, \delta)},\\
&\|T\|_{k,0,\nu}\sim\|T\|_{k,0,\nu}'\sim \sum_{j=0}^k \sup_{M\smallsetminus\{S\}}\big[w_M^{-p+q+j+\nu}\big|\gnabla^j T|_g\big];
\end{aligned}
\end{equation*}
here $\widehat{w_{M,\Phi}}:=w_M\big(\Phi(0)\big)$ are sample values of the weight function. Let 
\begin{equation*}
C^{k,\alpha}_{p,q}(M;S):=\Big\{T^{i_1...i_p}_{j_1... j_q}\ \big|\ \|T\|_{k,\alpha}<\infty\} \text{\ \ and\ \ }
C^{k,\alpha,\nu}_{p,q}(M;S):=\Big\{T^{i_1...i_p}_{j_1... j_q}\ \big|\ \|T\|_{k,\alpha,\nu}<\infty\Big\}
\end{equation*}
be the corresponding (weighted) function spaces.

We have the following weighted elliptic estimates for the vector Laplacian $L_g$ and the linearized Lichnerowicz operator $\L_g$. The latter can schematically be represented by 
$$\L_gf=\Delta_gf-\tfrac{1}{8}R(g)\,f+h\, f$$
with $h:M\to \R$  smooth. 

\begin{thm}\label{ell-M;S}
Let $\alpha\in(0,1)$. If  vector field $X\in C^{0,0,\nu}(M;S)$ is such that  $L_g X\in C^{k,\alpha,\nu+2}(M;S)$ then $X\in C^{k+2,\alpha,\nu}(M;S)$ and
$$\|X\|_{k+2,\alpha, \nu}\cle \|L_g X\|_{k,\alpha, \nu+2} +\|X\|_{0,0,\nu}.$$
Likewise, if  function $f\in C^{0,0,\nu}(M;S)$ satisfies 
$\L_g f\in C^{k, \alpha, \nu+2}(M;S)$ then $f\in C^{k+2,\alpha, \nu}(M;S)$
and 
$$\|f\|_{k+2,\alpha, \nu}\cle\|\L_g f\|_{k,\alpha, \nu+2} +\|f\|_{0,0,\nu}.$$
\end{thm}
The proof of this theorem is analogous to the proof of Theorem \ref{unif-ell}: one uses the scaling properties of the vector Laplacian and the linearized Licherowicz operator together with the basic interior elliptic estimates and various norm equivalences. We omit the details, but point out that in the case of the linearized Licherowicz operator we use the fact that $h\in C^{k,\alpha,2}(M;S)$.
 
\subsection{Weighted function spaces on $(M_0, g_0)$} We also need function spaces which keep track of the decay of tensors on $(M_0, g_0)$ with respect to  the radial function of the asymptotia. The construction of these spaces is analogous to that used for the spaces $C^{k,\alpha,\nu}(M_\e)$ and $C^{k,\alpha,\nu}(M;S)$.

The weight function $w_0:M_0\to \R$ below can be viewed as a smooth extension of the radial function $|\so^{-1}(.)|$. We define $w_0$ to be the unique continuous function which is constant  away from $\so\big(\R^3\smallsetminus \bar{B}_C\big)$ and satisfies
$$(w_0)\circ(\so):x\mapsto w_{0,\R}(|x|)$$
for the function $w_{0,\R}$ of Section \ref{weight-sec}. 

Next we introduce the special atlases $\C_0$ and $\C_0'$ for $M_0$. These contain two types of charts:
\begin{enumerate}
\item {\sc Charts in the asymptotia.}  Let $P=\so(x_P)\in\Oo_{4C}$. To this point we associate the chart $\so_P:=\so\circ \H_P$, where $\H_P:B_1\to \R^3$ is defined by $\H_P(x)=x_P+\frac{|x_P|}{2}x$.  By construction $\big(w_0\circ \so_P\big)(x)=\big|\H_P(x)\big|$ for all $x\in B_1$. The methods used in the proof of Proposition \ref{prop-covers} show that the metrics 
$$\go_P=\tfrac{4}{|x_P|^2}\, \so_P^*g_0,\ \ \ \ \ P\in\Oo_{4C},$$
are uniformly close to the Euclidean metric $\delta$. We also consider the restrictions $\so_P'$ of $\so_P$ to the ball $B_\frac{1}{2}$. 
\medbreak
\item {\sc Charts away from the asmyptotia.} We cover $M_0\smallsetminus \Oo_{4C}$ with the charts $\so_n$, $1\le n\le N$, introduced in Section \ref{atlases-sec}. We also consider the restrictions $\so_n'$ of these charts to $B_\frac{1}{2}$, and the metrics $g_{-n}=\so_n^*g_0$ discussed before. 
\end{enumerate}
Define the atlas $\C_0$ to be the collection of charts $\so_P$ for $P\in \Oo_{4C}$ and charts $\so_n$, $1\le n\le N$. Furthermore, define the H\o lder norm $\|T\|_{k,\alpha, \nu}$ by
\begin{equation*}
\|T\|_{k,\alpha}:=\sup_{\Phi\in\C_0}\|\Phi^* T\|_{C^{k,\alpha}(B_1, \delta)},\ \ \ \ \|T\|_{k,\alpha, \nu}:=\|w_0^\nu T\|_{k,\alpha}.
\end{equation*}
The atlas $\C_0'$ and the norm $\|T\|_{k,\alpha, \nu}'$ are defined analogously using restrictions of $\Phi\in\C_0$ to $B_{\frac{1}{2}}$.  
One can easily extend the results of Proposition \ref{prop-covers} to prove the following norm equivalences:
\begin{equation*}
\begin{aligned}
&\|T\|_{k,\alpha, \nu}\sim\sup_{\Phi\in \C_0}\widehat{w_{0,\Phi}}^\nu\|\Phi^* T\|_{C^{k,\alpha}(B_1, \delta)}\sim \|T\|_{k,\alpha,\nu}'\sim \sup_{\Phi\in\C_0'}\widehat{w_{0,\Phi}}^\nu\|\Phi^* T\|_{C^{k,\alpha}(B_\frac{1}{2}, \delta)},\\
&\|T\|_{k,0,\nu}\sim\|T\|_{k,0,\nu}'\sim \sum_{j=0}^k \sup_{M_0}\big[w_0^{-p+q+j+\nu}\big|\gonabla^j T|_{g_0}\big];
\end{aligned}
\end{equation*}
here $\widehat{w_{0,\Phi}}:=w_0\big(\Phi(0)\big)$ are sample values of the weight function. Let 
\begin{equation*}
C^{k,\alpha}_{p,q}(M_0;\infty):=\Big\{T^{i_1...i_p}_{j_1... j_q}\ \big|\ \|T\|_{k,\alpha}<\infty\} \text{\ \ and\ \ }
C^{k,\alpha,\nu}_{p,q}(M_0;\infty):=\Big\{T^{i_1...i_p}_{j_1... j_q}\ \big|\ \|T\|_{k,\alpha,\nu}<\infty\Big\}
\end{equation*}
denote the corresponding (weighted) function spaces.
We use these function spaces when studying the vector Laplacian $L_{g_0}$. The following theorem is analogous to Theorem \ref{unif-ell}.
\begin{thm}\label{ell-M;infty}
Let $\alpha\in(0,1)$. If  $X\in C^{0,0,\nu}_{1,0}(M_0;\infty)$ is such that $L_{g_0} X\in C^{k,\alpha,\nu+2}_{1,0}(M_0;\infty)$ then $X\in C^{k+2,\alpha,\nu}_{1,0}(M_0;\infty)$ and
$$\|X\|_{k+2,\alpha, \nu}\cle \|L_{g_0} X\|_{k,\alpha, \nu+2} +\|X\|_{0,0,\nu}.$$
\end{thm}
To prove the theorem one uses the scaling properties of the vector Laplacian together with the basic interior elliptic estimates and various norm equivalences. The details are left to the reader. 

\subsection{Weighted function spaces on $(\R^3\smallsetminus\{0\}, \delta)$} It is avoidable, but highly convenient to use function spaces on $(\R^3\smallsetminus\{0\}, \delta)$ with weight function $r:x\mapsto |x|$. In a sense these function spaces blend the usefulness of $C^{k,\alpha,\nu}(M;S)$ near the origin with that of $C^{k,\alpha,\nu}(M_0;\infty)$ ``near" $\infty$. 

The special atlas $\C$ consists of charts of the form 
$$\H_P:B_1\to \R^3\smallsetminus\{0\},\ \  \H_P(x)=P+\tfrac{|P|}{2}x,\ \ \text{where}\  P\in \R^3\smallsetminus\{0\}.$$ 
We define the (weighted) H\"older norm in the following manner:
\begin{equation*}
\|T\|_{k,\alpha}:=\sup_{\H_P\in\C}\|\H_P^* T\|_{C^{k,\alpha}(B_1, \delta)},\ \ \ \ \|T\|_{k,\alpha, \nu}:=\|r^\nu T\|_{k,\alpha}.
\end{equation*}
The atlas $\C'$ and the norm $\|T\|_{k,\alpha, \nu}'$ are defined analogously using the restrictions of $\H_P\in\C$ to $B_{\frac{1}{2}}$. 
One can easily extend the results of Proposition \ref{prop-covers} and prove the norm equivalences
\begin{equation*}
\begin{aligned}
&\|T\|_{k,\alpha, \nu}\sim\sup_{\H_P\in \C}|P|^\nu\,\|\H_P^* T\|_{C^{k,\alpha}(B_1, \delta)}\sim \|T\|_{k,\alpha,\nu}'\sim \sup_{\H_P\in\C'}|P|^\nu\,\|\H_P^* T\|_{C^{k,\alpha}(B_\frac{1}{2}, \delta)},\\
&\|T\|_{k,0,\nu}\sim\|T\|_{k,0,\nu}'\sim \sum_{j=0}^k \sup_{\R^3\smallsetminus\{0\}}\big[r^{-p+q+j+\nu}\big|\dnabla^j T|_\delta\big].
\end{aligned}
\end{equation*}
The corresponding weighted function spaces are
\begin{equation*}
C^{k,\alpha,\nu}_{p,q}(\R^3;0,\infty):=\Big\{T^{i_1...i_p}_{j_1... j_q}\ \big|\ \|T\|_{k,\alpha,\nu}<\infty\Big\}.
\end{equation*}
We use these function spaces when studying the vector Laplacian $L_\delta$ and the scalar Laplacian $\Delta_\delta$. The following theorem is analogous to Theorem \ref{unif-ell}.
\begin{thm}\label{ell-R3}
Let $\alpha\in(0,1)$. If $X\in C^{0,0,\nu}_{1,0}(\R^3;0,\infty)$ is such that $L_\delta X\in C^{k,\alpha,\nu+2}_{1,0}(\R^3;0,\infty)$  then 
$X\in C^{k+2,\alpha,\nu}_{1,0}(\R^3;0,\infty)$  and 
$$\|X\|_{k+2,\alpha, \nu}\cle \|L_\delta X\|_{k,\alpha, \nu+2} +\|X\|_{0,0,\nu}.$$
If  function $f\in C^{0,0,\nu}(\R^3;0,\infty)$ satisfies  
$\Delta_\delta f\in C^{k, \alpha, \nu+2}(\R^3;0,\infty)$ then 
$f\in C^{k+2,\alpha, \nu}(\R^3;0,\infty)$ and 
$$ \|f\|_{k+2,\alpha, \nu}\cle\|\Delta_\delta f\|_{k,\alpha, \nu+2} +\|f\|_{0,0,\nu}.$$
\end{thm}
The details of the proof are left to the reader.

\section{Repairing the Momentum Constraint}

We start by estimating the extent to which our approximate data $(M_\e, g_\e, K_\e)$ fail to satisfy the momentum constraint. With the notational conventions of Section \ref{mu:glue}  this constraint reduces to $\div_{g_\e} \mu_\e=0$. To accommodate the needs of our discussion later on, we not only prove an estimate for $(\div_{g_\e}\mu_\e)^{\sharp}$ but also a more general result regarding $\mu_\e$.

\begin{prop}\label{div:est} \hfill
\begin{enumerate}
\item We have $\|\mu_\e\|_{k,\alpha,-1}\cle 1$.
\item Let $U_\e:=w_\e^{-1}\left(\tfrac{\sqrt{\e}}{4}, +\infty\right)\subseteq M_\e$. We have $\|\mu_\e\|_{k,\alpha,-2;U_\e}\cle 1$.
\item Let $V_\e:=w_\e^{-1}\left(0,12\sqrt{\e}\right)\subseteq M_\e$. We have 
$\|\mu_\e\|_{k,\alpha,0;V_\e}\cle \e$. 
\item Finally, $\|(\div_{g_\e}\mu_\e)^{\sharp}\|_{k,\alpha,\nu}\cle \e^{(\nu/2)-1}.$ 
\end{enumerate}
\end{prop}

\begin{proof} We use the definition of the H\o lder norms and study the pullback of $\mu_\e$ along $\Phi\in\C_\e$. There are three cases to consider, based on the type of $\Phi$. We begin with the most interesting, {\sc Type $\G$}. 

Let   
$\Phi=\s_P$, $P=\s(x_P)\in \G_\e$. To simplify our computation, we set 
$$\chi_1(x):=\chi\left(\tfrac{6}{\sqrt{\e}}\left|\H_P(x)\right|-2\right) \ \
\text{and}\ \ \chi_2(x):=(1-\chi)\left(\tfrac{3}{4\sqrt{\e}}\left|\H_P(x)\right|-2\right),$$
so that 
$$\s_P^*\mu_\e=\e\chi_1\cdot\left[\H_P^*\se^*\mu_0\right]+\chi_2\cdot\left[\H_P^*\sM^*\mu\right].$$
While the cut-offs $\chi_1$, $\chi_2$ have a potential to contribute a significant amount to the derivatives of $\mu_\e$ they do not cause any actual trouble: the methods used in the proof  of Proposition \ref{prop-covers}, along with 
$$|d\chi_1|^2+|d\chi_2|^2\neq 0 \Longrightarrow |x_P|\sim \sqrt{\e},$$
show that for each multi-index $\beta$ the derivatives 
$\partial^\beta\chi_1$ and  $\partial^\beta \chi_2$ are bounded on $B_1$ independently of $\e$ and $P$. We estimate the norm of the term involving $\mu_0=K_0$ by using \eqref{asymptK}. As $\H_P$ dilates by a factor of $\frac{|x_P|}{2}$ it follows that   (for each multi-index $\beta$)
\begin{equation}\label{muest1}
\left|\partial^\beta \left(\H_P^*\se^*\mu_0\right)_{ij}\right|\cle |x_P|^2\frac{1}{|x_P|^{|\beta|+2}}|x_P|^{|\beta|}\cle 1
\end{equation}
independently of $P$.
Likewise, if $(\sM^*\mu)_{ij;\beta}$ denotes the combination of the partial derivatives of $(\sM^*\mu)_{ij}$ corresponding to a multi-index $\beta$ then
\begin{equation}\label{muest2}
\left|\partial^\beta \left(\H_P^*\sM^*\mu\right)_{ij}\right|\cle |x_P|^2\Big((\sM^*\mu)_{ij;\beta}\circ \H_P\Big) |x_P|^{|\beta|}\cle |x_P|^{2+|\beta|}\cle |x_P|^2.
\end{equation}
These estimates yield
\begin{equation*}
\left\|\s_P^*\mu_\e\right\|_{C^{k,\alpha}(B_1,\delta)}\cle \e+\widehat{w_{\e, \s_P}}^2 \text{\ \ \ independently of $P\in \G_\e$}.
\end{equation*}

For charts $\Phi$ of {\sc Type $M$} one can easily show that 
$\left\|\Phi^*\mu_\e\right\|_{C^{k,\alpha}(B_1,\delta)}\cle \wef^2\cle \e+\wef^2$;  the case of  {\sc Type $M_0$} is similar: $\left\|\Phi^*\mu_\e\right\|_{C^{k,\alpha}(B_1,\delta)}\cle \e\cle \e+\wef^2$. 

The first claim of our proposition now follows from 
$\e+\wef^2\cle \wef$ and 
$$\wef^{-1}\left\|\Phi^*\mu_\e\right\|_{C^{k,\alpha}(B_1, \delta)}\cle 1
\text{\ \ \ independently of $\Phi\in \C_\e$}.$$ 
The second (respectively  third) claim follows from 
$\e+\wef^2\cle \wef^2$ (respectively, $\e+\wef^2\cle \e$) and 
$\wef^{-2}\left\|\Phi^*\mu_\e\right\|_{C^{k,\alpha}(B_1, \delta)}\cle 1$ (respectively $\left\|\Phi^*\mu_\e\right\|_{C^{k,\alpha}(B_1, \delta)}\cle \e$)
which hold independently of $\Phi\in\C_{\e;U_\e}$ (respectively 
$\Phi\in\C_{\e;V_\e}$). 

A similar method is used to estimate the size of  $\left(\div_{g_\e}\mu_\e\right)^{\sharp}$. Note that in order for $\Phi^*\left(\div_{g_\e}\mu_\e\right)\neq 0$ the chart $\Phi$ needs to be of {\sc Type $\G$}  with   $|x_P|\in\left[\tfrac{\sqrt{\e}}{4},12\sqrt{\e}\right]$. 
It follows from $\grad_{\s_P^*g_\e}=\frac{4}{|x_P|^2}\grad_{g_P}$, $\div_g \mu=0$ and  $\div_{g_0}\mu_0=0$  that 
$$\s_P^* \div_{g_\e}\mu_\e=\e\frac{4}{|x_P|^2}\left(\grad_{g_P}\chi_1\right)\into\left(\H_P^*\se^*\mu_0\right) +\frac{4}{|x_P|^2} \left(\grad_{g_P}\chi_2\right)\into \left(\H_P^*\sM^*\mu\right)
$$
where $V\into \omega$ denotes the $1$-form $\omega(\cdot,V)$. Given that the metrics $g_P$ are uniformly close to the Euclidean metric $\delta$ (see Proposition \ref{prop-covers}), relations \eqref{muest1}, and \eqref{muest2}  imply 
$$\left\|\s_P^*\left((\div_{g_\e}\mu_\e)^{\sharp}\right)\right\|_{C^{k,\alpha}(B_1,\delta)}\cle \frac{1}{\e}\left\|\H_P^*\se^*\mu_0\right\|_{C^{k,\alpha}(B_1,\delta)}+\frac{1}{\e^2}\left\|\H_P^*\sM^*\mu\right\|_{C^{k,\alpha}(B_1,\delta)}\cle \frac{1}{\e}$$
independently of $P$ with $|x_P|\in\left[\tfrac{\sqrt{\e}}{4},12\sqrt{\e}\right]$.  In conclusion,  we have 
$$\widehat{w_{\e, \s_P}}^\nu \left\|\s_P^*\left((\div_{g_\e}\mu_\e)^{\sharp}\right)\right\|_{C^{k,\alpha}(B_1,\delta)}\cle \e^{(\nu/2)-1}$$
and $\|(\div_{g_\e}\mu_\e)^{\sharp}\|_{k,\alpha,\nu}\cle \e^{(\nu/2)-1}.$
\end{proof}

To repair the momentum constraint we perturb $\mu_\e$ so that the resulting trace-free symmetric $2$-tensor is divergence-free. One classic way of doing this \cite{IMP} involves solving the linear PDE 
\begin{equation}\label{VectLapl:eqn}
L_\e X_\e=(\div_{g_\e} \mu_\e)^{\sharp}.
\end{equation}

Observe that \eqref{VectLapl:eqn} has a solution in $C^{k+2,\alpha,\nu}_{1,0}(M_\e)$ for each $k, \alpha, \nu$.  
Indeed $L_\e:H^{k+2}(M_\e)\to H^{k}(M_\e)$ (viewed as an operator between ordinary Sobolev spaces) is self-adjoint and elliptic; in particular, 
$$\mathrm{Im}(L_\e)=\mathrm{Ker}(L_\e)^{\perp}$$
with respect to the  $L^2$-pairing. If $Y\in\mathrm{Ker}(L_\e)$ then integration by parts yields $Y\in\mathrm{Ker}(\D_\e)$ and 
$$\int_{M_\e}g_\e\left((\div_{g_\e} \mu_\e)^{\sharp}, Y\right)=-\int_{M_\e} g_\e\left(\D_\e^* \mu_\e, Y\right)=-\int_{M_\e} g_\e\left(\mu_\e, \D_\e Y\right) =0.$$
It follows that  $(\div_{g_\e} \mu_\e)^{\sharp}\in \mathrm{Im}(L_\e)\subseteq H^{k}(M_\e)$ and that there exists a solution $X_\e\in H^{k+2}(M_\e)$ of \eqref{VectLapl:eqn}. By  the Sobolev Embedding Theorem we see that  $X_\e\in C^{0,0,\nu}_{1,0}(M_\e)$. Elliptic regularity, Theorem \ref{unif-ell}, shows that $X_\e\in C^{k+2,\alpha,\nu}_{1,0}(M_\e)$. 

We use $X_\e$ to make a small perturbation of $\mu_\e$ and repair the momentum constraint. Thus, it is crucial that we have a control on the size of $X_\e$. We achieve this by proving the following uniformity property of the family of operators $L_\e$.  

\begin{prop}\label{VL:MAIN}
Let $\nu\in\left(\tfrac{3}{2}, 2\right)$ and let $\e>0$ be sufficiently small.  We have 
$$\|X\|_{0,0,\nu}\cle \|L_\e X\|_{0,0,\nu+2}$$
independently of smooth vector fields $X$ on $M_\e$. 
\end{prop}

\begin{proof}
We assume opposite: that there exist $\e_j\downarrow 0$ ($j\in\N$) and vector fields $X_j$ on $M_{\e_j}$ such that 
\begin{equation}\label{VL:mainassumpt}
\|X_j\|_{0,0,\nu}=1 \text{\ \ and\ \ } \|L_{\e_j} X_j\|_{0,0,\nu+2}\to 0.
\end{equation}
Equivalently, the first property of $X_j$ can be written as 
$$\max_{M_{\e_j}} w_{\e_j}^{-1+\nu} \left|X_j\right|_{g_{\e_j}}=1.$$
Let $P_j\in M_{\e_j}$ be the points at which these maxima are reached. Consider the sequence $\left( w_{\e_j}(P_j)\right)_{j\in\N}$. One of the following holds:
\begin{description}
\item[{\sc Case $M\smallsetminus\{S\}$}] There exists a subsequence of $(P_j)_{j\in\N}$, which we may without loss of generality assume is $(P_j)_{j\in\N}$ itself, and a number $c_M>0$ such that 
$$w_{\e_j}(P_j)\ge c_M \text{\ \ for all\ \ } j\in \N.$$
\item[{\sc Case $\R^3\smallsetminus\{0\}$}] There exists a subsequence of $(P_j)_{j\in\N}$, which we may without loss of generality assume is $(P_j)_{j\in\N}$ itself, such that 
$$w_{\e_j}(P_j)\to 0, \ \ \ \frac{\e_j}{w_{\e_j}(P_j)}\to 0.$$
\item[{\sc Case $M_0$}] There exists a subsequence of $(P_j)_{j\in\N}$, which we may without loss of generality assume is $(P_j)_{j\in\N}$ itself, and a number $c_{M_0}>0$ such that 
$$w_{\e_j}(P_j)\le c_{M_0}\e_j \text{\ \ for all\ \ }j\in\N.$$
\end{description}
In each of the three cases we use the sequence $(X_j)_{j\in\N}$ to construct a non-trivial vector field on the indicated manifold. By construction the vector field is in a particular weighted H\o lder space and is in this kernel of a vector Laplacian. We obtain a contradiction by arguing that there is no such vector field. The reasoning is similar in each of the three cases. In the interest of brevity we present only one of the cases in full detail, namely {\sc Case $M\smallsetminus\{S\}$}. 

\underline{\sc Obtaining the contradiction in Case $M\smallsetminus\{S\}$.} Our first step is the construction of a vector field on $M\smallsetminus\{S\}$ with peculiar properties; we refer to  this step as the \emph{Exhaustion Argument}.

Let $\widetilde{D}_1\subset D_1\subset \widetilde{D}_2\subset D_2\subset.....\subset M\smallsetminus\{S\}$ be a sequence of compact subsets of $M\smallsetminus \{S\}$ such that $$\bigcup_{n=1}^\infty \mathrm{Int}(D_n)=M\smallsetminus \{S\}.$$
Without loss of generality we may assume that 
$w_M\big|_{\widetilde{D}_1}\ge c_M$ and that the restriction of the quotient map $\pMi$ to $D_1$,
$$\pMi: D_1\to  M_{\e_j},$$
is an embedding with $\pMi^* g_{\e_j}=g$ for all $j\in \N$. 
The restrictions of vector fields $\X_j:=\pMi^* X_j$ to $D_1$ satisfy
\begin{equation}\label{VL:nontriv}
\sup_{\widetilde{D}_1}\left| \X_j\right|_g\ge c,
\ \ \ \left| \X_j\right|_g\le w_M^{1-\nu},
\ \ \ \left| L_g \X_j\right|_g\le c_j \cdot w_M^{-1-\nu} \text{\ \ and\ \ } 
\lim_{j\to \infty} c_j=0,
\end{equation}
where $c>0$ is a constant independent of $j\in \N$, and where $c_j:=\|L_{g_{\e_j}} X_j\|_{0,0,\nu+2}$.

Now consider the interior elliptic estimate 
$$\|\X_j\|_{H^2(\widetilde{D}_1,g)}\le C_1\left(\|L_g \X_j\|_{L^2(D_1,g)}+\|\X_j\|_{L^2(D_1,g)}\right)$$
in Sobolev spaces with respect to the metric $g$. 
It follows from \eqref{VL:nontriv} and the uniform boundedness of $w_M=w_{\e_j}\circ \pMi$  on $D_1$ away from $0$ that  the sequences $(\X_j)_{j\in \N}$ and $(L_g \X_j)_{j\in \N}$ are bounded in $L^2(D_1,g)$. Consequently, $(\X_j)_{j\in\N}$ is bounded as a sequence in $H^2(\widetilde{D}_1,g)$. From the  Rellich Lemma and the Sobolev Embedding Theorem we see that there is a subsequence of $(\X_j)_{j\in \N}$ which is convergent in $C^0(\widetilde{D}_1,g)$.
We extract and relabel the subsequences $(\e_j)_{j\in \N}$, $(\X_j)_{j\in \N}$  to get 
$$ \X_j \to Y_1 \text{ \ \ in\ \ } C^0(\widetilde{D}_1,g).$$
Note that $Y_1\neq 0$ and $|Y_1|_g \le w_M^{1-\nu}$ on $\widetilde{D}_1$ due to \eqref{VL:nontriv}. 

We now repeat the process: eliminating finitely many terms of $(\e_j)_{j\in \N}$ and $(\X_j)_{j\in \N}$  we ensure that $\pMi^* g_{\e_j}=g$ on $D_2$.  The interior elliptic estimate
$$\|\X_{j}\|_{H^2(\widetilde{D}_2,g)}\le C_2\left(\|L_g \X_{j}\|_{L^2(D_2,g)}+\|\X_j\|_{L^2(D_2,g)}\right)$$
implies the existence of a subsequence of $(X_j)_{j\in\N}$ whose pullback  is convergent in $C^0(\widetilde{D}_2, g)$. As above, we extract and relabel this subsequence so to have $$\X_j\to Y_2\text{\ \ in\ \ }C^0(\widetilde{D}_2,g).$$ 
Since $Y_2$ is a subsequential limit of the sequence which defines $Y_1$ we have $Y_2\big|_{\widetilde{D}_1}=Y_1$. We also note that $|Y_2|_g\le w_M^{1-\nu}$ and that  \eqref{VL:nontriv} holds on $D_2$.
 
The process described above gives rise to an iterative construction of vector fields $Y_n\in C^0(\widetilde{D}_n,g)$, $n\in \N$ such that 
$$Y_n\big|_{\widetilde{D}_{n-1}}=Y_{n-1}, \ \ \ |Y_n|_g\le w_M^{1-\nu}.$$
Define the vector field $Y$ on $M\smallsetminus\{S\}$ by 
$$Y\big|_{\widetilde{D}_n}=Y_n.$$
We have $Y\in C^{0,0,\nu}_{1,0}(M;S)$, $\|Y\|_{0,0,\nu}\le 1$ and $Y\neq 0$. 

The punch-line of the Exhaustion Argument is that 
$L_g Y=0$ \emph{on $M\smallsetminus\{S\}$}. By elliptic regularity it suffices to show that 
$L_g Y=0$ \emph{weakly}. To that end, let $\xi$ be a compactly supported vector field on $M\smallsetminus\{S\}$. Let $m\in \N$ be such that 
$\mathrm{supp}(\xi)\subseteq \widetilde{D}_m$. Since $Y\big|_{\widetilde{D}_m}=Y_m$ and since $L_g$ is formally self adjoint we have 
$$\begin{aligned}
\left|\int_{M\smallsetminus\{S\}} g\left(Y, L_g\xi\right)\dvol_g\right|=&
\left|\int_{\widetilde{D}_m} g\left(Y_m, L_g\xi\right)\dvol_g\right|\\
=&\lim_{j\to \infty} \left|\int_{\widetilde{D}_m} g\left(\X_j, L_g\xi\right)\dvol_g\right|
=
\lim_{j\to \infty} \left|\int_{\widetilde{D}_m} g\left(L_g\X_j, \xi\right)\dvol_g\right|\\
\le &\|\xi\|_{C^0(\widetilde{D}_m, g)}\, \mathrm{vol}_g(\widetilde{D}_m)\cdot\lim_{j\to \infty} \|L_g \X_j\|_{C^0(\widetilde{D}_m, g)}.
\end{aligned}$$
Since  $w_M^{-1-\nu}$ is bounded on $\widetilde{D}_m$ there is a constant $c(\widetilde{D}_m)$ such that  $$\|L_g \X_j\|_{C^0(\widetilde{D}_m, g)}\le c_j\, c(\widetilde{D}_m) \text{\ \ for\ all\ \ } j\in \N.$$ In particular, we have  
$\displaystyle{\lim_{j\to \infty}} \|L_g \X_j\|_{C^0(\widetilde{D}_m, g)}=0$
and  $L_g Y=0$ on $M\smallsetminus\{S\}$. 

It is now important to notice that Theorem \ref{ell-M;S} implies $Y\in C^{2,0, \nu}_{1,0}(M;S)$. Consequently, there is a constant $\tilde{c}$ such that 
\begin{equation}\label{VL:est2}
\left|\nabla Y\right|_g\le \tilde{c}\cdot w_M^{-\nu}.
\end{equation}
We now show that the existence of the vector field $Y$ described above is a contradiction.

We start by showing that $L_g Y=0$ \emph{weakly on $M$}. Let $\xi$ be a vector field on $M$ and let $B^M_r$ be a geodesic ball of (small) radius $r$ centered at $S$. To understand $\int_M g(Y, L_g\xi)\dvol_g$ we estimate  $\int_{B^M_r} g(Y,L_g\xi)\dvol_g$ and $\int_{M\smallsetminus B^M_r} g(Y,L_g\xi)\dvol_g$ individually.

For the first integral we take the advantage of $|Y|_g\le w_M^{1-\nu}$ to see that for some constant $c_1(\xi)$ (independent of $r$) we have
$$\left|\int_{B^M_r} g(Y,L_g\xi)\dvol_g\right|\le c_1(\xi)\cdot r^{4-\nu}.$$
On the other hand, integration by parts and the fact that $L_g Y=0$ on $M\smallsetminus\{S\}$ imply
$$\left|\int_{M\smallsetminus B^M_r} g(Y, L_g \xi)\dvol_g\right|\le \int_{\partial B^M_r} |\D_g\xi\left(Y, \mathbf{n}\right)+\D_g Y\left(\xi, \mathbf{n}\right)|\dvol_g,$$
where $\mathbf{n}$ denotes a unit normal to the geodesic sphere $\partial B^M_r$. Point-wise estimates for $|Y|_g$ and $|\nabla Y|_g$ (see \eqref{VL:est2}) yield
\begin{gather*}
\int_{\partial B^M_r} |\D_g \xi (Y, \mathbf{n})|\dvol_g\le \int_{\partial B^M_r} c_2(\xi) r^{1-\nu}\dvol_g\le c_3(\xi) r^{3-\nu},\\
\int_{\partial B^M_r} |\D_g Y (\xi, \mathbf{n})|\dvol_g\le c_4(\xi) r^{2-\nu}
\end{gather*}
for some constants $c_2(\xi), c_3(\xi), c_4(\xi)$ independent of $r$. 
Combining all of the above we obtain
$$\left|\int_M g(Y, L_g \xi)\dvol_g\right|\le c_5(\xi) r^{2-\nu}$$
for some constant $c_5(\xi)$ independent of $r$. 
Since $r$ is arbitrary we may take the limit as $r\to 0$; as a result we obtain
$$\int_M g(Y, L_g \xi)\dvol_g=0.$$
In other words, we have that $L_g Y=0$ weakly on $M$. By elliptic regularity  $L_g Y=0$ strongly and $Y$ is smooth on all of $M$. 

Integrating by parts we further see that 
$$0=\int_M g(L_g Y, Y)\dvol_g=\int_M \left|\D_g Y\right|^2_g\dvol_g,$$
i.e.~that $\D_g Y=0$ on $M$. This is a contradiction since $Y\neq 0$ and there are no non-trivial conformal Killing vector fields on $M$. 

\underline{\sc Obtaining the contradiction in Case $\R^3\smallsetminus\{0\}$.} In this case we have $P_j\in \G_{\e_j}$ for all but finitely many $j\in\N$. To be able to employ the Exhaustion Argument we need to do some re-scaling. More precisely, we blow up the gluing region $\G_{\e_j}$ by a factor of $w_j:=w_{\e_j}(P_j)$ and re-scale the metrics and vector fields correspondingly.  

Consider the dilation 
$\H_j:x\mapsto w_j\cdot x$ of $\R^3$ and define
$$\Omega_j:=(\s\circ \H_j)^{-1}\left(\G_{\e_j}\right)=\left\{x\in \R^3\ \Big|\ 4C\frac{\e_j}{w_j}< |x|< \frac{C^{-1}}{4w_j}\right\}.$$
This choice  is motivated by the fact that the points 
$Q_j\in \Omega_j$ with $P_j=\s\circ \H_j(Q_j)$ satisfy $|Q_j|=1$.
Since $w_j\to 0$ and $\frac{\e_j}{w_j}\to 0$ as $j\to \infty$ each compact subset  $D\subseteq \R^3\smallsetminus\{0\}$ is contained in all but finitely many $\Omega_j$. 

Next consider the metrics 
$$g^\Omega_j:=\tfrac{1}{w_j^2}\cdot (\s\circ \H_j)^* g_{\e_j}$$
on $\Omega_j$.
The methods used in the proof of Proposition \ref{prop-covers} (see  \eqref{unif-eucl} for details) show that for each compact subset $D\subseteq \R^3\smallsetminus\{0\}$ there exists a constant $c(D)$ such that 
$$|g^\Omega_j-\delta|\le c(D)\left(\frac{\e_j}{w_j}+w_j^2\right)$$
for all but finitely many $j$.  It follows from $w_j\to 0$ and $\frac{\e_j}{w_j}\to 0$ that $g^\Omega_j$ converges to the Euclidean metric $\delta$  uniformly on $D$ as $j\to \infty$. A similar line of reasoning shows that $g^\Omega_j$ converges to $\delta$  in the $C^k(D,\delta)$-norm for any compact subset $D\subseteq \R^3\smallsetminus\{0\}$.

Define $\X_j:= w_j^\nu\left(\s\circ \H_j\right)^* X_j$. We claim there exists  a sequence $(c_j)_{j\in \N}$ which converges to zero  and a constant $c>0$ with the following properties.
\begin{itemize}
\item The supremum over the unit sphere $S^2\subseteq \R^3\smallsetminus\{0\}$ satisfies $\sup_{S^2}|\X_j|_\delta\ge c$ for all $j\in \N$.
\medbreak
\item For each compact subset $D\subseteq \R^3\smallsetminus\{0\}$ there is $j_0\in \N$ such that for all $j\ge j_0$ and all $Q\in D$ we have $|\X_j (Q)|_\delta\le \tfrac{1}{c}|Q|^{1-\nu}$. 
\medbreak
\item For each compact subset $D\subseteq \R^3\smallsetminus\{0\}$ there exist $j_0\in \N$ and a constant $c(D)$ such that for all $j\ge j_0$  we have
 $|L_{g_j^\Omega}\X_j|_{g_j^\Omega}\le c_j\cdot c(D)$ on $D$.
\end{itemize}
These properties are consequences of the normalization $Q_j\in S^2$ for all $j\in \N$, the scaling identities $$|\X_j|_{g_j^\Omega}=w_j^{\nu-1}|X_j|_{g_{\e_j}}\circ \H_j,\ \ \ w_{\e_j}\circ\s\circ \H_j(Q)=w_j|Q|,\ \ \ L_{g_j^\Omega}=w_j^2 L_{(\s\circ \H_j)^*g_{\e_j}}$$ and the convergence $g_j^\Omega\to \delta$ on compact subsets. To illustrate the proofs of these properties we verify the last inequality.  

Let $c_j:=\|L_{\e_j}X_j\|_{0,0,\nu+2}$;
by assumption $c_j\to 0$ as $j\to \infty$. Further, let $D\subseteq \R^3\smallsetminus\{0\}$ be a compact set and let $Q\in D$. Since $g_j^\Omega\to \delta$ on $D$ there is $j_0\in \N$ such that 
$$\left|L_{g_j^\Omega}\X_j(Q)\right|_{g_j^\Omega}=
 w_j^{1+\nu}\left|L_{\e_j} X_j(Q)\right|_{g_{\e_j}}\le 
w_j^{1+\nu}\cdot \frac{\|L_{\e_j}X_j\|_{0,0,\nu+2}}{(\s\circ\H_j)^*w_{\e_j}^{1+\nu}}(Q)
= c_j |Q|^{-1-\nu}$$
for all $j\ge j_0$. The constant $c(D)$ can be chosen to be $\left(\displaystyle{\inf_{Q\in D}}|Q|\right)^{-1-\nu}$. 

We now apply the Exhaustion Argument to the vector fields $\X_j$ and obtain a vector field $Y\neq 0$ on $\R^3\smallsetminus\{0\}$ such that 
\begin{equation}\label{VLdecay}
Y\in C^{0,0,\nu}_{1,0}(\R^3;0,\infty) \text{\ \ and\ \ }  L_\delta Y=0.
\end{equation}
The Exhaustion Argument here is essentially identical to the one presented in {\sc Case $M\smallsetminus\{S\}$}, except in the integration by parts step which proves that $L_\delta Y=0$ weakly on $\R^3\smallsetminus\{0\}$. The key difference in the step is that, for a given test vector field $\xi$ and compact set $\widetilde{D}_m\supseteq \mathrm{supp}(\xi)$, we use $\|g_j^\Omega -\delta\|_{C^2(\widetilde{D}_m,\delta)}\to 0$ and its consequence
\begin{equation}\label{VL:passtoEucl}
\lim_{j\to \infty}\left\|L_{g_j^\Omega}\xi - L_\delta \xi\right\|_{C^0(\widetilde{D}_m,\delta)}=0.
\end{equation}
For clarity we do the integration explicitly here:
$$\begin{aligned}
\left|\int_{\R^3\smallsetminus\{0\}} \delta\left(Y, L_\delta\xi\right)\dvol_\delta\right|=&
\left|\int_{\widetilde{D}_m} \delta\left(Y_m, L_\delta\xi\right)\dvol_\delta\right|\\
=&\lim_{j\to \infty} \left|\int_{\widetilde{D}_m} g_j^\Omega\left(\X_j, L_{g_j^\Omega}\xi\right)\dvol_{g_j^\Omega}\right|
=
\lim_{j\to \infty} \left|\int_{\widetilde{D}_m} g_j^\Omega\left(L_{g_j^\Omega}\X_j, \xi\right)\dvol_{g_j^\Omega}\right|\\
\le &\lim_{j\to \infty} \left(\|\xi\|_{C^0(\widetilde{D}_m, g_j^\Omega)}\,\|L_{g_j^\Omega} \X_j\|_{C^0(\widetilde{D}_m, g_j^\Omega)}\,\mathrm{vol}_{g_j^\Omega}(\widetilde{D}_m)\right)\\
\le &\|\xi\|_{C^0(\widetilde{D}_m, \delta)}\, \mathrm{vol}_\delta(\widetilde{D}_m)\cdot\lim_{j\to \infty} \left(c_j c(\widetilde{D}_m)\right)=0.
\end{aligned}$$

The next step is to show that the existence of $Y$ described above is a contradiction. We have discovered two ways to do this.  One way is to use spherical coordinates and spherical harmonics to \emph{explicitly} compute the kernel of $L_\delta$; this approach is addressed in the Appendix. The approach we take here is to show that $Y$ is a conformal Killing vector field on $(\R^3,\delta)$ which decays at $\infty$; the explicit knowledge of all the Euclidean conformal Killing vector fields shows that the existence of such a $Y$ is impossible. 

We first observe that $L_\delta Y=0$ weakly (and hence strongly) \emph{on $\R^3$}.  Indeed, Theorem \ref{ell-R3} shows that 
$Y\in C^{2,0,\nu}_{1,0}(\R^3;0,\infty)$ and, in particular,  
\begin{equation}\label{VL:R3est}
|Y|\le c_1\,r^{1-\nu},\ \ \ |\nabla Y|\le c_1\,r^{-\nu}
\end{equation} 
for some constant $c_1>0$. 
These estimates ensure that the integration-by-parts argument of {\sc Case $M\smallsetminus\{S\}$} carries over with no changes. 
Next let $B_\rho$ (resp.~ $S^2_\rho$) denote the Euclidean ball (resp.~ sphere) of radius $\rho$ centered at the origin. 
Consider the integration by parts formula
$$0=\int_{B_\rho}\delta\left(L_\delta Y, Y\right) \dvol_\delta= \int_{S^2_\rho}\D_\delta Y\left(Y, \mathbf{n}\right)\dvol_\delta
+\int_{B_\rho}\left|\D_\delta Y\right|^2\dvol_\delta$$
in which $\D_\delta$ denotes the Euclidean conformal Killing operator and in which $\mathbf{n}$ denotes the appropriately oriented unit normal to  $S^2_\rho$. 
It follows from \eqref{VL:R3est} that 
$$\left|\int_{S^2_\rho}\D_\delta Y\left(Y, \mathbf{n}\right)\dvol_\delta\right|\le c_2\rho^{3-2\nu}$$
for some constant $c_2$. By assumption $\nu>\tfrac{3}{2}$ and so 
$$\int_{\R^3}\left|\D_\delta Y\right|^2\dvol_\delta=-\lim_{\rho\to +\infty} \int_{S^2_\rho}\D_\delta Y\left(Y, \mathbf{n}\right)\dvol_\delta=0.$$
We conclude that  $\D_\delta Y=0$ on $\R^3$. In other words, we have that $Y\in C^{2,0,\nu}_{1,0}(\R^3;0,\infty)$ is a non-zero conformal Killing vector field on $\R^3$ which decays at $\infty$. To see that this is a contradiction one can appeal to a generalization of a theorem of Christodoulou \cite[Prop.~13]{IMP} or simply recall that  the space of conformal Killing vector fields on $\R^3$ is spanned by 
coordinate (translation) vector fields $\mathbf{e}_i$, rotation vector fields $x^i\mathbf{e}_j-x^j\mathbf{e}_i$, dilation vector field $\sum x^i\mathbf{e}_i$, and the special vector fields 
$2x^j\left(\sum x^i\mathbf{e}_i\right)-\left(\sum (x^i)^2\right)\mathbf{e}_j$, none of which decay at $\infty$. This completes the proof in {\sc Case $\R^3\smallsetminus\{0\}$}.

\underline{\sc Obtaining the contradiction in Case $M_0$.} To apply the Exhaustion Argument we need to re-scale $M_0$ ``back to its original size". We consider the regions 
$$\Omega_j:=M_0\smallsetminus\Oo_{\left(2\sqrt{\e_j}\right)^{-1}},$$
which  $\pei$ map diffeomorphically into $M_{\e_j}$. Since $\e_j\to 0$ each compact subset $D\subseteq M_0$ is contained in all but finitely many $\Omega_j$. If necessary, we eliminate finitely many terms of $(\e_j)_{j\in\N}$ so that $P_j\in \pei(\Omega_j)$  and $\pei^*g_{\e_j}=\e_j^2 g_0$ for all $j$. Let $Q_j\in M_0$ be such that $P_j=\pei(Q_j)$. Finally, consider the vector fields 
$$\mathbf{X}_j:=\e_j^\nu \left(\pei\right)^*X_j \text{\ \ on\ \ } \Omega_j.$$
One easily verifies the following properties of the vector fields $\X_j$:
\begin{enumerate}
\item There is a constant $c_1$ independent of $j$ such that 
$\left|\X_j\right|_{g_0}\le c_1$ point-wise on $\Omega_j$,
\item $\left|\X_j\right|_{g_0}\le w_0^{1-\nu}$ point-wise on the ``asymptotic" region $\Omega_j\cap \Oo_{2C}$,
\item $\left|\X_j(Q_j)\right|_{g_0}\ge c_{M_0}^{1-\nu}$ and 
\item There is a constant $c_2$ independent of $j$ with $\left|L_{g_0}\X_j\right|_{g_0}\le c_2\left\|L_{\e_j}X_j\right\|_{0,0,\nu+2}$ point-wise on $\Omega_j$. 
\end{enumerate}
Applying the Exhaustion Argument to the vector fields $(\X_j)_{j\in \N}$ we obtain a \emph{non-zero} vector field $Y\in C^{0,0,\nu}_{1,0}(M_0;\infty)$ for which $L_{g_0}Y=0$.  
It follows from Theorem \ref{ell-M;infty} that $Y\in C^{2,0,\nu}_{1,0}(M_0;\infty)$ and 
\begin{equation}\label{VL:est3}
\left|Y\right|_{g_0}\le \tilde{c}\, w_0^{1-\nu},\ \ \left|\nabla Y\right|_{g_0}\le \tilde{c}\, w_0^{-\nu}
\end{equation}
for some constant $\tilde{c}$. With these estimates in hand we apply the integration-by-parts argument used in {\sc Case $\R^3\smallsetminus\{0\}$} to show that $Y$ is a conformal Killing vector field on $(M_0,g_0)$.  In addition, we see from \eqref{VL:est3} that $Y\neq 0$ must ``decay at infinity". This situation is impossible according to a generalization of a theorem of Christodoulou (see \cite[Prop.~13]{IMP}). Thus we have reached our final contradiction!
\end{proof}
The following important consequence of Proposition \ref{VL:MAIN} is immediate from Theorem \ref{unif-ell}.
\begin{thm}
Let  $\alpha\in(0,1)$ and $\nu\in\left(\tfrac{3}{2}, 2\right)$. If $\e>0$ is sufficiently small, then 
$$\|X\|_{k+2,\alpha,\nu}\cle \|L_\e X\|_{k,\alpha,\nu+2}$$
independently of smooth vector fields $X$ on $M_\e$. 
\end{thm}

Recall that our strategy for repairing the momentum constraint involves solving the equation \eqref{VectLapl:eqn}, for which there always exists a solution $X_\e$. Proposition \ref{div:est}  and the previous theorem now provide a weighted H\"older estimate for $X_\e$. More precisely, if  $\nu\in \left(\tfrac{3}{2}, 2\right)$ and $\e$ is small then 
\begin{equation}\label{xe:est}
\|X_\e\|_{k+2,\alpha,\nu} \cle \e^{\nu/2}.
\end{equation}

Consider the perturbation $\me$ of $\mu_\e$ defined by 
$$\me :=\mu_\e+\D_{\e} X_\e, \mathrm{\ \ \ \ \ where\ \ \ \ } L_\e X_\e=(\div_{g_\e}\mu_\e)^{\sharp}.$$
Since $\D_\e$ maps into the subspace of symmetric and trace-free $2$-tensors, the tensor $\me$ is itself symmetric and trace-free with respect to $g_\e$. Furthermore, the choice of $X_\e$ ensures that $\me$ is also divergence-free:
$$\left(\div_{g_\e} \me\right)^{\sharp}=\left(\div_{g_\e} \mu_\e\right)^{\sharp}-\D_\e^*\D_\e X_\e=L_\e X_\e-L_\e X_\e=0.$$
Thus, the pair of tensors $(g_\e, \me+\tfrac{\t}{3}g_\e)$ satisfies the momentum constraint. The following proposition compares $\me$ with $\mu_\e$ and shows that $\me$ is indeed a small perturbation of $\mu_\e$.
\begin{prop} \label{corr:est}
If  $\nu\in\left(\tfrac{3}{2}, 2\right)$ then:
\begin{enumerate}
\item $\left\|\me-\mu_\e\right\|_{k,\alpha,\nu-2}\cle \e^{\nu/2}$.
\item $\left\|\left|\me\right|^2_{g_\e}-\left|\mu_\e\right|^2_{g_\e}\right\|_{k,\alpha,\nu+1}\cle \e^{\nu/2}$.
\item $\left\|\left|\me\right|^2_{g_\e}-\left|\mu_\e\right|^2_{g_\e}\right\|_{k,\alpha,\nu;U_\e}\cle \e^{\nu/2}$, where $U_\e:=w_\e^{-1}\left(\tfrac{\sqrt{\e}}{4}, +\infty\right)\subseteq M_\e$. 
\end{enumerate}
\end{prop}
\begin{proof} The claim regarding the size of $\me-\mu_\e=\D_\e X_\e$ follows immediately from \eqref{xe:est} and Proposition \ref{De}. The second estimate is a consequence of Proposition \ref{Holder:props}, the earlier estimate for $\D_\e X_\e$, and the first claim of Proposition \ref{div:est}. Indeed 
\begin{equation*}
\begin{aligned}
\left\|\left|\me\right|^2_{g_\e}-\left|\mu_\e\right|^2_{g_\e}\right\|_{k,\alpha,\nu+1}\cle&\left\|g_\e\left(\D_\e X_\e, \mu_\e+\D_\e X_\e\right)\right\|_{k,\alpha,\nu+1}\cle \left\|\D_\e  X_\e\otimes \left(\mu_\e+\D_\e X_\e\right)\right\|_{k,\alpha, \nu-3}\\
\cle&\left\|\D_\e X_\e\right\|_{k,\alpha,\nu-2}\left(\left\|\mu_\e\right\|_{k,\alpha,-1}+\left\|\D_\e X_\e\right\|_{k,\alpha,-1}\right)\\
\cle&\e^{\nu/2}\left(1+\e^{1-\nu}\left\|\D_\e X_\e\right\|_{k,\alpha,\nu-2}\right)\cle \e^{\nu/2}\left(1+\e^{1-\nu/2}\right)\cle \e^{\nu/2};
\end{aligned}
\end{equation*}
note that the above relies on the fact that $-1<\nu-2$ and $1-\tfrac{\nu}{2}>0$. The third estimate is proved similarly using
$$\|\D_\e X_\e\|_{k,\alpha, -2;U_\e}\cle \e^{-\nu/2} \|\D_\e X_\e\|_{k,\alpha, \nu-2;U_\e}\cle 1$$
and the second claim of Proposition \ref{div:est}.
\end{proof}

With little additional work one can prove the following estimates for $|\me|^2_{g_\e}$ itself; these estimates are crucial in our analysis of the Lichnerowicz equation. 

\begin{prop}\label{me:est}
Let  $\nu\in\left(\tfrac{3}{2},2\right)$. 
\begin{enumerate}
\item We have $\left\|\left|\me\right|^2_{g_\e}\right\|_{k,\alpha, 2}\cle 1$.
\item If  $U_\e:=w_\e^{-1}\left(\tfrac{\sqrt{\e}}{4}, +\infty\right)\subseteq M_\e$, then $\left\|\left|\me\right|^2_{g_\e}\right\|_{k,\alpha, 0; U_\e}\cle 1$.
\item If $V_\e:=w_\e^{-1}\left(0,12\sqrt{\e}\right)\subseteq M_\e$, then $\left\|\left|\me\right|^2_{g_\e}\right\|_{k,\alpha, \nu+1; V_\e}\cle \e^{\nu-1}$.
\item If $W_\e:=w_\e^{-1}\left(\tfrac{\sqrt{\e}}{4},12\sqrt{\e}\right)\subseteq M_\e$, then $\left\|\left|\me\right|^2_{g_\e}\right\|_{k,\alpha, \nu; W_\e}\cle \e^{\nu/2}$.
\end{enumerate}
\end{prop}

\begin{proof}
The proofs of all four claims are simple manipulations of estimates in Propositions \ref{div:est} and \ref{corr:est}, using properties of Proposition \ref{Holder:props}.  To illustrate the arguments we prove the last  claim of our proposition. \begin{equation*}
\begin{aligned}
\left\|\left|\me\right|^2_{g_\e}\right\|_{k,\alpha,\nu;W_\e}\cle 
&\e^{(\nu-4)/2}\left\|\left|\mu_\e\right|^2_{g_\e}\right\|_{k,\alpha,4;W_\e}+
\left\|\left|\me\right|^2_{g_\e}-\left|\mu_\e\right|^2_{g_\e}\right\|_{k,\alpha, \nu;W_\e}\\
\cle &\e^{(\nu-4)/2}\left\|\mu_\e\otimes\mu_\e\right\|_{k,\alpha,0;W_\e}+\e^{\nu/2}\cle \e^{(\nu-4)/2}\cdot \e^2+\e^{\nu/2}\cle \e^{\nu/2}.
\end{aligned}
\end{equation*}
The remaining proofs are left to the reader. 
\end{proof}

We conclude this section with an interpretation of  
$\left\|\me-\mu_\e\right\|_{k,\alpha,\nu-2}\cle \e^{\nu/2}$  in terms of the ordinary and the scaled point-particle limit properties.  Consider a compact subset  $\mathbf{K}\subseteq M\smallsetminus\{S\}$ and the corresponding embedding  
$i_\e:\mathbf{K}\to M_\e$. Since $1\cle w_\e$ on $i_\e(\mathbf{K})$ the norm  equivalence \eqref{Holder:equiv}  yields 
\begin{equation}\label{ordinarylim:mu}
\left\|(i_\e)^*\me -\mu\right\|_{C^k(\mathbf{K},g)}=O(\e^{\nu/2}) \text{\ \ as\ \ }\e\to 0.
\end{equation}
On the other hand, consider a compact set $\mathbf{K}\subseteq M_0$ and the corresponding embedding  $\iota_\e:\mathbf{K}\to M_\e$. Note that $w_\e \sim \e$ on $\iota_\e(\mathbf{K})$. Proposition \ref{corr:est} and the norm equivalence \eqref{Holder:equiv} imply  $$\e^{(2+k)+(\nu-2)}\left|\genabla^k\left(\me-\mu_\e\right)\right|_{g_\e}=\e^{\nu-2}\left|\gonabla^k\left(\me-\mu_\e\right)\right|_{g_0}\cle \e^{\nu/2}.$$
Since $\left\|\left(\iota_\e\right)^*\left(\tfrac{\t}{3}g_\e\right)\right\|_{C^k(\mathbf{K},g_0)}=\t\e^2$ we have 
\begin{equation}\label{scaledlim:mu}
\left\|\tfrac{1}{\e}\left(\iota_\e\right)^*\left(\me+\tfrac{\t}{3}g_\e\right)-K_0\right\|_{C^k(\mathbf{K},g_0)}=O\left(\e^{1-\nu/2}\right) \text{\ \ as\ \ }\e\to 0.
\end{equation} 
The limits \eqref{ordinarylim:mu} and \eqref{scaledlim:mu} show that the \emph{approximate data} $\left(g_\e, \me+\tfrac{\t}{3}g_\e\right)$ satisfy the ordinary and the scaled point-particle limit properties.

\section{The Lichnerowicz Equation}
While the data $\left(g_\e, \me+\tfrac{\t}{3}g_\e\right)$ satisfy the momentum constraint, they need not satisfy the Hamiltonian constraint. To address the situation we apply the conformal method; in other words, we make a suitable conformal change of the data which repairs the Hamiltonian constraint and yet preserves the momentum constraint.  One such change is  
\begin{equation} \label{conf:format}
g_\e\to \phi_\e^4 g_\e, \ \ \ \me+\tfrac{\tau}{3}g_\e\to \phi_\e^{-2}\me+\tfrac{\t}{3}\phi_\e^4g_\e,
\end{equation} 
where $\phi_\e$ is a positive solution of the  Lichnerowicz equation \eqref{LICHN}. In light of \eqref{ordinarylim:mu} and \eqref{scaledlim:mu} we see that the ordinary and the scaled point-particle limit properties hold for the resulting data provided the solution of the Lichnerowicz equation satisfies $\phi_\e\approx 1$ (in some sense of the word). To avoid notational confusion we let $\phi_0$ be the constant function $\phi_0\equiv 1$.

Most of the work in this section is dedicated to establishing the existence of a solution $\phi_\e$ of \eqref{LICHN} such that $\phi_\e-\phi_0$ satisfies desirable H\"older estimates. The first step in this analysis  is to understand the extent to which the  function  $\phi_0$ fails to be a solution of \eqref{LICHN}. More specifically, we need H\"older estimates  of $\Ne(\phi_0)$ where $\Ne$ denotes  the (non-linear) Lichnerowicz operator
$$\Ne(\phi):=\Delta_{g_\e}\phi - \tfrac{1}{8}R(g_\e)\phi+\tfrac{1}{8}|\me|_{g_\e}^2\phi^{-7}+\left(\tfrac{\Lambda}{4}-\tfrac{\t^2}{12}\right)\phi^5.$$

\begin{prop}\label{NE:est}
If $\nu\in\left(\tfrac{3}{2},2\right)$ then
$\|\Ne(\phi_0)\|_{k,\alpha, \nu+1}\cle \e^{\nu/2}$.
\end{prop}

\begin{proof} We show that 
\begin{equation}\label{ne:basicest}
\wef^{\nu+1}\left\|\Phi^*\Ne(\phi_0)\right\|_{C^{k,\alpha}(B_1,\delta)}\cle \e^{\nu/2}
\end{equation} 
independently of $\Phi\in\C_\e$. We distinguish three cases, based on whether a suitable pullback of $(g_\e,\mu_\e)$ matches with  $(g,\mu)$, with $(\e^2g_0,\e\mu_0)$, or with neither.

{\sc The case of $\mathrm{Im}(w_\e\circ \Phi) \subseteq \left[8\sqrt{\e},\infty\right)$.} Here we use the fact that $(g,K)$ satisfies the Hamiltonian constraint with cosmological constant $\Lambda$. Since $\mu$ is trace-free we have $|K|_g^2=\tfrac{\t^2}{3}+|\mu|_g^2$ and  
$$R(g)-|\mu|_g^2+\tfrac{2}{3}\t^2-2\Lambda=0.$$
Consequently, the expression for $\Phi^*\Ne(1)$ simplifies to:
$$\Phi^*\Ne(\phi_0)=-\tfrac{1}{8}\left(R(\Phi^*g_\e)-|\Phi^*\mu_\e|^2_{\Phi^*g_\e}+\tfrac{2}{3}\t^2-2\Lambda\right)+\tfrac{1}{8}\Phi^*\left(\left|\me\right|^2_{g_\e}-\left|\mu_\e\right|^2_{g_\e}\right)=\tfrac{1}{8}\Phi^*\left(\left|\me\right|^2_{g_\e}-\left|\mu_\e\right|^2_{g_\e}\right).$$ 
The claim \eqref{ne:basicest} is now immediate from Proposition \ref{corr:est}. 

{\sc The case of $\mathrm{Im}(w_\e\circ \Phi) \subseteq \left(0,\tfrac{\sqrt{\e}}{2}\right]$.} We first simplify the expression for $\Phi^*\Ne(1)$ using the fact that $(\e^2 g_0, \e K_0)$ satisfies the Hamiltonian constraint with no cosmological constant:
$$\Phi^*\Ne(\phi_0)=\tfrac{1}{8}\Phi^*\left(\left|\me\right|^2_{g_\e}-\left|\mu_\e\right|^2_{g_\e}\right)+\left(\tfrac{\Lambda}{4}-\tfrac{\t^2}{12}\right).$$
By assumption we have $w_\e\circ \Phi \cle\sqrt{\e}$ and therefore
\begin{equation}\label{ne:constterms}
\wef^{\nu+1}\left\|\tfrac{\Lambda}{4}-\tfrac{\t^2}{12}\right\|_{C^{k,\alpha}(B_1,\delta)}\cle \e^{\frac{\nu+1}{2}}\cle \e^{\frac{\nu}{2}}.
\end{equation}
Once again, the claim \eqref{ne:basicest} follows from Proposition \ref{corr:est}.

{\sc The case of $\mathrm{Im}(w_\e\circ \Phi)\cap \left(\tfrac{\sqrt{\e}}{2},8\sqrt{\e}\right)\neq \emptyset$.} 
More specifically, we have  
$\Phi=\s_P$ for some $P=\s(x_P)$ with $\tfrac{\sqrt{\e}}{2}<|x_P|<8\sqrt{\e}$. 
Our strategy is to estimate each individual term of $\Phi^*\Ne(\phi_0)$. 

Recall from the proof of Proposition \ref{prop-covers} that 
$$\|g_P-\delta\|_{C^{k+3}}\cle \frac{\e}{|x_P|}+|x_P|^2.$$
This estimate implies $\|R(g_P)\|_{C^{k,\alpha}(B_1,\delta)}\cle \sqrt{\e}$ and 
$$\wef^{\nu+1}\|\Phi^*R(g_\e)\|_{C^{k,\alpha}(B_1,\delta)}=|x_P|^{\nu+1}\tfrac{4}{|x_P|^2}\|R(g_P)\|_{C^{k,\alpha}(B_1,\delta)}\cle \e^{\nu/2}$$
independently of $P$ such that $\tfrac{\sqrt{\e}}{2}<|x_P|<8\sqrt{\e}$. To estimate $|\me|^2_{g_\e}$ we use the last claim of Proposition \ref{me:est} which, by virtue of $\mathrm{Im}\Phi\subseteq w_\e^{-1}\left(\tfrac{\sqrt{\e}}{4}, 12\sqrt{\e}\right)$, implies
$$\wef^{\nu+1}\|\Phi^*|\me|^2_{g_\e}\|_{C^{k,\alpha}(B_1,\delta)}\cle \e^{\nu/2}.$$
Since \eqref{ne:constterms} continues to hold in this case, our proof is complete. 
\end{proof}

The main ingredient in our study of the Lichnerowicz equation is  the uniform invertibility of the linearizations 
$$\L_\e:=\Delta_{g_\e}-\frac{1}{8}R(g_\e)-\frac{7}{8}|\me|^2_{g_\e}+5\left(\tfrac{\Lambda}{4}-\tfrac{\t^2}{12}\right)$$
of $\Ne$ at $\phi_0$. To match the notation here with that used in Section \ref{operators:section} we let 
$$h_\e:=-\frac{7}{8}|\me|^2_{g_\e}+5\left(\tfrac{\Lambda}{4}-\tfrac{\t^2}{12}\right).$$
It was claimed in Section \ref{operators:section} that $h_\e$ satisfy an estimate of the form 
$\|h_\e\|_{k,\alpha, 2}\cle 1$.
This fact is a consequence of Proposition \ref{me:est} and $\left\|\tfrac{\Lambda}{4}-\tfrac{\t^2}{12}\right\|_{k,\alpha,2}\cle 1$.

Our approach to proving the uniform invertibility of $\L_\e$ is analogous to the approach taken in our analysis of the vector Laplacian. As the proofs of the main technical propositions, Proposition \ref{VL:MAIN} and Proposition \ref{LICHN:MAIN}, share many common features  we avoid repetition whenever possible and refer the reader to the proof of Proposition \ref{VL:MAIN}. 

\begin{prop}\label{LICHN:MAIN}
Let $\nu\in\left(\tfrac{3}{2}, 2\right)$ and let $\e>0$ be sufficiently small.  We have 
$$\|\phi\|_{0,0,\nu-1}\cle \|\L_\e \phi\|_{0,0,\nu+1}$$
independently of smooth functions $\phi$ on $M_\e$. 
\end{prop}

\begin{proof}
We assume opposite: that there are $\e_j\downarrow 0$ ($j\in\N$) and (smooth) functions $\phi_j$ on $M_{\e_j}$ with 
\begin{equation*}
\|\phi_j\|_{0,0,\nu-1}=1 \text{\ \ and\ \ } \|\L_{\e_j} \phi_j\|_{0,0,\nu+1}\to 0.
\end{equation*}
Let $P_j\in M_{\e_j}$ be the points at which
$$w_{\e_j}(P_j)^{-1+\nu} \left|\phi_j(P_j)\right|=1.$$
Depending on  the nature of the sequence $\left(w_{\e_j}(P_j)\right)_{j\in\N}$ we distinguish three cases: {\sc Case $M\smallsetminus\{S\}$}, {\sc Case $\R^3\smallsetminus\{0\}$} and {\sc Case $M_0$}; this is done \emph{exactly} as in the proof of Proposition \ref{VL:MAIN}. 

\underline{\sc Obtaining the contradiction in Case $M\smallsetminus\{S\}$.} The strategy in this case is to use the  sequence $\left(\phi_j\right)_{j\in\N}$ to construct a non-zero function $\psi:M\to \R$ which is in the kernel of 
$$\L_g:=\Delta_g-\frac{1}{8}R(g)-\frac{7}{8}|\mu|^2_g+5\left(\tfrac{\Lambda}{4}-\tfrac{\t^2}{12}\right);$$
the existence of such a function contradicts the {\sc Injectivity assumption} on $(M,g,K)$. 

Let $D\subseteq M\smallsetminus\{S\}$ be a compact subset which contains $w_M^{-1}\left([c_M,+\infty)\right)$. By eliminating finitely many $\phi_j$ from consideration we may assume that the functions
$$\psi_j:D\to \R,\ \ \ \psi_j:=\phi_j\circ \pMi$$
are well-defined and satisfy 
\begin{equation}\label{LICHN:1}
\sup_D|\psi_j|\ge c_1 \text{\ \ and\ \ } |\psi_j|\le w_M^{1-\nu}
\end{equation}
for some constant $c_1>0$ independent of $D$ and $j$. Since  $w_M\big|_D$ is bounded away from zero the sequence $\left(\psi_j\right)_{j\in \N}$ is bounded in $C^0(D,g)$ and $\|\pMi^*\lei \phi_j\|_{C^0(D,g)}\to 0$. 

Let $j$ be large enough so that $\pMi^* g_{\e_j}=g$, $\pMi^* \mu_{\e_j}=\mu$. One easily computes 
$$\left|\pMi^* \lei \phi_j - \L_g \psi_j\right|=\frac{7\left|\psi_j\right|}{8}\left|\left(\left|\mei\right|^2_{g_{\e_j}}\circ \pMi\right)-\left|\mu\right|^2_g\right|.$$
Proposition \ref{corr:est} implies 
$$\left\|\pMi^* \lei \phi_j - \L_g \psi_j\right\|_{C^0(D,g)}\le c_2(D)\e_j^{\nu/2},\ \ j\gg 1$$
for some constant $c_2(D)$ which potentially depends on $D$. It now follows that 
\begin{equation}\label{LICHN:2}
\lim_{j\to +\infty}\|\L_g \psi_j\|_{C^0(D,g)} =0.
\end{equation}

The properties \eqref{LICHN:1} and \eqref{LICHN:2} allow us to apply the Exhaustion Argument (see the proof of Proposition \ref{VL:MAIN}) to functions $\psi_j$. We conclude that there exists a \emph{non-zero} function $\psi\in C^{0,0,\nu-1}(M;S)$  such that 
$$\L_g \psi=0 \text{\ \ on\ \ }M\smallsetminus\{S\}.$$
We see from Theorem \ref{ell-M;S} that $\psi\in C^{2,0,\nu-1}(M;S)$  and  consequently
$$\left|\nabla \psi\right|_g\le c_3\cdot w_M^{-\nu}$$
for some constant $c_3$. This knowledge of the ``blow-up" rate of $\psi$ at $S$ allows us to show that $\psi$ satisfies $\L_g \psi =0$ \emph{weakly on $M$}. 
The computation is analogous to that in the proof of Proposition \ref{VL:MAIN}. The conclusion is that $\psi$ is a non-zero smooth function on $M$ which is in the kernel of $\L_g$.

On the other hand, we see from the Hamiltonian constraint that  $R(g)-|\mu|_g^2+\tfrac{2}{3}\t^2-2\Lambda=0$ and
\begin{equation*}
\L_g=\Delta_g-\left(|\mu|^2_g+\tfrac{1}{3}\t^2-\Lambda\right)=\Delta_g-\left(|K|^2_g-\Lambda\right).
\end{equation*}
It follows from the {\sc Injectivity assumption} that $\L_g$ has trivial kernel. We have reached a contradiction which  completes the argument in {\sc Case $M\smallsetminus\{S\}$}.

\underline{\sc Obtaining the contradiction in Case $\R^3\smallsetminus\{0\}$.}
The strategy in this case is to use the  sequence $\left(\phi_j\right)_{j\in\N}$ to construct a non-zero harmonic function $\psi:\R^3\to \R$ which is in $C^{0,0,\nu-1}(\R^3;0,\infty)$.
The existence of such a function is, by the Maximum Principle, a contradiction.

Adopt the notation used in the proof of Proposition \ref{VL:MAIN}, {\sc Case $\R^3\smallsetminus\{0\}$}. Define 
$$\psi_j:\Omega_j\to \R\ \ \ \psi_j:=w_j^{\nu-1}\,\left(\phi_j\circ \s\circ \H_j\right).$$
The scaling identity $w_{\e_j}\circ\s\circ \H_j(Q)=w_j|Q|$ implies the following properties of the functions $\psi_j$:
\begin{itemize}
\item The supremum over the unit sphere $S^2\subseteq \R^3\smallsetminus\{0\}$ satisfies $\sup_{S^2}\left|\psi_j\right|\ge 1$.
\item If $Q\in \Omega_j$ then $\left|\psi_j(Q)\right|\le |Q|^{1-\nu}$.
\end{itemize}
Consider the operators  
$$\L_j:=\Delta_{g_j^\Omega}-\tfrac{1}{8}R(g_j^\Omega)-\tfrac{7}{8}\,w_j^2\left(\left|\mei\right|^2_{g_{\e_j}}\circ\s\circ\H_j\right)+5\,w_j^2\left(\tfrac{\Lambda}{4}-\tfrac{\t^2}{12}\right);$$
these operators are of interest since
$$\L_j \psi_j=w_j^{\nu+1}\,\left[ (\L_{\e_j}\phi_j)\circ \s\circ \H_j\right]$$
for all $j\in \N$. Using our assumption on $\L_{\e_j}\phi_j$ one can easily show that 
$$\lim_{j\to +\infty}\|\L_j \psi_j\|_{C^0(D,g_j^\Omega)}=0$$
on each compact subset $D\subseteq \R^3\smallsetminus\{0\}$. 
Notice that, in some sense, the sequence of operators $(\L_j)_{j\in\N}$ itself converges. More precisely, we claim that if $\eta$ is a test function on $\R^3\smallsetminus\{0\}$ then  
\begin{equation}\label{LICHN:passtoEucl}
\lim_{j\to \infty}\left\|\L_j\eta - 
\Delta_\delta \eta\right\|_{C^0\left(\R^3,\delta\right)}=0.
\end{equation}
Before we prove \eqref{LICHN:passtoEucl} we point out that this identity plays the same role in the overall proof of Proposition  \ref{LICHN:MAIN} as identity \eqref{VL:passtoEucl} plays in the proof of Proposition \ref{VL:MAIN}.

To prove \eqref{LICHN:passtoEucl} let $\eta$ be a test function on $\R^3\smallsetminus\{0\}$. We compute:
$$|\L_j\eta-\Delta_\delta\eta|\le
|\Delta_{g_j^\Omega}\eta-\Delta_\delta\eta|+
\tfrac{1}{8}\left|R(g_j^\Omega)\eta\right|+
\tfrac{7}{8}w_j^2\left(\left|\mei\right|^2_{g_{\e_j}}\circ\s\circ\H_j\right)|\eta|+
5\,w_j^2\left|\left(\tfrac{\Lambda}{4}-\tfrac{\t^2}{12}\right)\eta\right|.
$$
Since $g_j^\Omega\to \delta$ on $\mathrm{supp}(\eta)$ uniformly with all the derivatives, we have that both $\|\Delta_{g_j^\Omega}\eta-\Delta_\delta\eta\|_{C^0\left(\R^3,\delta\right)}$ and $\left\|R(g_j^\Omega)\eta\right\|_{C^0\left(\R^3,\delta\right)}$ converge to $0$ as $j\to \infty$. 
The second and the third estimate of Proposition \ref{me:est}, together with \eqref{normeq:restr}, imply that for some constant $\tilde{c}>0$ and all 
$j\in \N$ at least one of the following two inequalities holds:
$$|\mei|^2_{g_{\e_j}}\le \tilde{c},\ \ \ \ \ w_{\e_j}^{\nu+1} |\mei|^2_{g_{\e_j}}\le \tilde{c}\,\e_j^{\nu-1}.$$
Therefore if $Q\in\mathrm{supp}(\eta)$ then 
$w_{\e_j}\circ\s\circ \H_j(Q)=w_j|Q|$ and at least one of the estimates
$$\left(|\mei|^2_{g_{\e_j}}\circ\s\circ\H_j\right)(Q)\le \tilde{c}(\eta),\ \ \ \ \ w_j^{\nu+1} \left(|\mei|^2_{g_{\e_j}}\circ\s\circ\H_j\right)(Q)\le \tilde{c}(\eta)\,\e_j^{\nu-1},$$
where $\tilde{c}(\eta)$ is a constant depending only on $\eta$. We can re-write these estimates jointly as 
$$w_j^2\left(\left|\mei\right|^2_{g_{\e_j}}\circ\s\circ\H_j\right)|\eta| \le\, 
\tilde{c}(\eta)\cdot\max_{\R^3}|\eta|\cdot\max\left\{w_j^2,\left(\tfrac{\e_j}{w_j}\right)^{\nu-1}\right\}.$$
Since both $w_j, \tfrac{\e_j}{w_j}\to 0$ in {\sc Case $\R^3\smallsetminus\{0\}$}, we conclude that 
$$\lim_{j\to \infty}
\left\|w_j^2\left(\left|\mei\right|^2_{g_{\e_j}}\circ\s\circ\H_j\right)\eta\right\|_{C^0\left(\R^3,\delta\right)}=0.$$
The convergence \eqref{LICHN:passtoEucl} is now immediate from 
$\displaystyle{\lim_{j\to \infty}}\left\|w_j^2\left(\tfrac{\Lambda}{4}-\tfrac{\t^2}{12}\right)\eta\right\|_{C^0\left(\R^3,\delta\right)}=0$.

We now apply the Exhaustion Argument $\left(\psi_j\right)_{j\in\N}$; we get a non-zero function $\psi\in C^{0,0,\nu-1}(\R^3;0,\infty)$ such that  $\Delta_\delta \psi=0$ on $\R^3\smallsetminus\{0\}$. Note that by Theorem \ref{ell-R3} we in fact have  $\psi\in C^{2,0,\nu-1}(\R^3;0,\infty)$. Consequently there is a constant $c$ such that  
$$|\psi|<c\,r^{1-\nu},\ \ \ |\dnabla \psi|<c\,r^{-\nu}$$
on $\R^3\smallsetminus\{0\}$. 
This control on the ``blow-up" of $\psi$ at the origin allows us to use an  integration-by-parts argument to show that $\Delta_\delta \psi=0$ \emph{weakly on $\R^3$}. We omit the integration details and refer the reader to the corresponding part of the proof of Proposition \ref{VL:MAIN}.
The overall conclusion is that $\psi$ is a non-zero harmonic function which, by virtue of $\psi\in C^{0,0,\nu-1}(\R^3;0,\infty)$, decays at $\infty$. This is a contradiction to the Maximum Principle. Our proof in {\sc Case $\R^3\smallsetminus\{0\}$} is now complete. 

\underline{\sc Obtaining the contradiction in Case $M_0$.} The strategy in this case is to construct a non-zero function $\psi\in C^{0,0,\nu-1}(M_0;\infty)$ which is in the kernel of the operator $\Delta_{g_0}-|\mu_0|^2_{g_0}$ on $M_0$; the existence of such a function is a contradiction to the Maximum Principle. 

Adopt the notation used in the proof of Proposition \ref{VL:MAIN}, {\sc Case $M_0$}. Define the functions 
$$\psi_j:\Omega_j\to \R,\ \ \ \ \psi_j:=\e_j^{\nu-1}\left(\phi_j\circ \pei\right)$$ and the operator
$$\L_j:=\Delta_{g_0}-\tfrac{1}{8}R(g_0)-\tfrac{7}{8}\e_j^2\left(\left|\mei\right|^2_{g_{\e_j}}\circ\pei\right)+5\e_j^2\left(\tfrac{\Lambda}{4}-\tfrac{\t^2}{12}\right).$$
One easily verifies the following properties of the sequence $\left(\psi_j\right)_{j\in\N}$.
\begin{itemize}
\item There is a constant $c_1$ independent of $j$ such that $\left|\psi_j\right|\le c_1$ point-wise on $\Omega_j$,
\item $\left|\psi_j\right|\le w_0^{1-\nu}$ point-wise on the ``asymptotic" region $\Omega_j\cap \Oo_{2C}$,
\item $\left|\psi_j(Q_j)\right|\ge c_{M_0}^{1-\nu}$ and 
\item There is a constant $c_2$ independent of $j$ such that 
$\left|\L_j\psi_j\right|\le c_2\left\|\L_{\e_j}\phi_j\right\|_{0,0,\nu+1}$ on $\Omega_j$. 
\end{itemize}
We now establish a property of the sequence of operators $\left(\L_j\right)_{j\in\N}$ which is analogous to \eqref{LICHN:passtoEucl}. Let $\eta$ be a test function on $M_0$, and let 
$$\L_0:=\Delta_{g_0}-\tfrac{1}{8}R(g_0)-\tfrac{7}{8}\left|\mu_0\right|^2_{g_0}.$$
To compare the actions of $\L_j$ and $\L_0$ on $\eta$ we use Proposition \ref{corr:est}, from which we see that
$$\left|\left(\left|\mei\right|^2_{g_{\e_j}}\circ \pei\right)-
\left|\e_j \mu_0\right|_{\e_j^2g_0}\right|\le c_3\e_j^{-1-\nu/2}$$
for some constant $c_3$.
This estimate implies  
$$\left|\e_j^2\left(\left|\mei\right|^2_{g_{\e_j}}\circ\pei\right)-\left|\mu_0\right|^2_{g_0}\right|
\to 0$$
uniformly on $\mathrm{supp}(\eta)$. 
The convergence 
$$\lim_{j\to \infty}\left\|\L_j\eta - \L_0\eta\right\|_{C^0(M_0,g_0)}=0$$
is now immediate.

The Exhaustion Argument applied to $(\psi_j)_{j\in\N}$ yields a \emph{non-zero}  function $\psi$ in $C^{0,0,\nu-1}(M_0;\infty)$ such that 
$\L_0 \psi=0$.
To understand the operator $\L_0$ better 
we use the Hamiltonian constraint 
$$R(g_0)-|\mu_0|^2_{g_0}=0$$
and re-write $\L_0$  as $\L_0=\Delta_{g_0}-|\mu_0|^2_{g_0}$. Since the Maximum Principle applies to such operators, there are no non-trivial functions in the kernel of $\L_0$ which decay at infinity. The existence of $\psi\in C^{0,0,\nu-1}(M_0;\infty)$ is therefore a contradiction. Our proof is now complete. 
\end{proof}

For $\alpha\in(0,1)$, $\nu\in \left(\tfrac{3}{2}, 2\right)$ and smooth $\phi$ we now have  $$\|\phi\|_{k+2,\alpha,\nu-1}\cle \|\L_\e\phi\|_{k,\alpha,\nu+1}$$
by virtues of Proposition \ref{LICHN:MAIN} and 
Theorem \ref{unif-ell}. It follows that  $\L_\e:C^{k+2,\alpha, \nu-1}(M_\e)\to C^{k,\alpha, \nu+1}(M_\e)$ is injective. In fact, we have a stronger result.

\begin{thm}\label{LICHN:full}
Let $\alpha\in(0,1)$, $\nu\in\left(\tfrac{3}{2},2\right)$ and let $\e$ be sufficiently small. The linearized Lichnerowicz operator 
$\L_\e:C^{k+2,\alpha,\nu-1}(M_\e)\to C^{k,\alpha,\nu+1}(M_\e)$
is invertible, and the norm of its inverse is bounded uniformly in $\e$.
\end{thm}

\begin{proof}
It remains to verify that 
$\L_\e:C^{k+2,\alpha,\nu-1}(M_\e)\to C^{k,\alpha,\nu+1}(M_\e)$ is surjective. As a self-adjoint elliptic operator on Sobolev spaces, $\L_\e:H^{k+2}(M_\e)\to H^k(M_\e)$ is Fredholm of index zero. The kernel of this operator consists of smooth functions and is therefore the same as the kernel of the operator $\L_\e:C^{k+2,\alpha,\nu-1}(M_\e)\to C^{k,\alpha,\nu+1}(M_\e)$, which is trivial. Thus  $\L_\e$ acting between Sobolev spaces is injective and, by Fredholm theory, surjective. It now follows from $$C^{k,\alpha,\nu+1}(M_\e)\subseteq H^k(M_\e),\ \  
H^{k+2}(M_\e)\subseteq  C^{0,0,\nu-1}(M_\e),$$ 
and elliptic regularity (Theorem \ref{unif-ell}) that  
$\L_\e:C^{k+2,\alpha,\nu-1}(M_\e)\to C^{k,\alpha,\nu+1}(M_\e)$ is also surjective.
\end{proof}

We solve the Lichnerowicz equation \eqref{LICHN} by interpreting it as a fixed point problem. A formal computation shows that the difference $\eta_\e:=\phi_\e - \phi_0$ between a solution $\phi_\e$ of the Lichnerowicz equation and the constant function 
$\phi_0\equiv 1$ satisfies 
$$\eta_\e=-(\L_\e)^{-1}\left(\Ne(\psi_0)+\Q_\e(\eta_\e)\right),$$
where  
$$
\Q_\e(\eta):=\,\Ne(\psi_0+\eta)-\Ne(\psi_0)-\L_\e\eta
=\tfrac{1}{8}|\me|^2\left((1+\eta)^{-7}-1+7\eta\right)-\left(\tfrac{\Lambda}{4}-\tfrac{\t^2}{12}\right)\left((1+\eta)^5-1-5\eta\right)
$$
is a ``quadratic error term".
Our strategy now is to show that the map
\begin{equation}\label{LICHN:contr}
\mathcal{P}_\e:\eta\mapsto -(\L_\e)^{-1}\Big(\Ne(\psi_0)+\Q_\e(\eta)\Big)
\end{equation}
is a contraction mapping from a small ball in $C^{k,\alpha,\nu-1}(M_\e)$ to itself. To execute this approach we need some estimates for $\Q_\e$. 
\begin{prop}\label{LICHN:Qest}
For a given $c>0$ and sufficiently small $\e$ there exists $c'>0$ independent of $\e$ such that for all $\eta_1, \eta_2\in C^{k,\alpha,\nu-1}(M_\e)$ with 
$$\|\eta_1\|_{k,\alpha,\nu-1},\|\eta_2\|_{k,\alpha,\nu-1}\le c\,\e^{\nu/2}$$
we also have 
$$\left\|\Q_\e(\eta_1)-\Q_\e(\eta_2)\right\|_{k,\alpha, \nu+1}\le c'\,\e^{\nu/2}\|\eta_1-\eta_2\|_{k,\alpha,\nu-1}.$$
\end{prop}

\begin{proof}
It follows from our assumption on the weighted H\"older norms of $\eta_1, \eta_2$ that 
\begin{equation}\label{LICHN:rndest}
\|\eta_1\|_{k,\alpha}\,,\,\|\eta_2\|_{k,\alpha}\le \tilde{c}\, \e^{1-\nu/2}
\end{equation} 
for some constant $\tilde{c}$ independent of $\e$. 
In particular, we see that  if $\e$ is small enough then  
$\left|\eta_1\right|,\,\left|\eta_2\right|<1$. Under such assumptions we may expand the algebraic terms in $\Q_\e(\eta_1)$ and  $\Q_\e(\eta_2)$ into binomial series. This process results in 
$$\|\Q_\e(\eta_1)-\Q_\e(\eta_2)\|_{k,\alpha, \nu+1}\le \sum_{j=2}^\infty a_j\cdot \|\eta_1^j-\eta_2^j\|_{k,\alpha,\nu+1},$$
where $a_j$ are some positive numbers (absolute values of linear combinations of binomial coefficients). 
Since 
$$\left\|\eta_1^j-\eta_2^j\right\|_{k,\alpha,\nu+1}\cle \|\eta_1-\eta_2\|_{k,\alpha,\nu-1} 
\sum_{i=0}^{j-1}\left\|\eta_1^{i}\eta_2^{j-i-1}\right\|_{k,\alpha,\nu-1},$$
we need to  bound the sum of the series 
$\sum_{i,j} a_{j+1}\,\e^{-\nu/2}\left\|\eta_1^{i}\eta_2^{j-i}\right\|$
uniformly in $\e$. 
An inductive argument which uses \eqref{LICHN:rndest} shows that 
$$\left\|\eta_1^{i}\eta_2^{j-i}\right\|_{k,\alpha,\nu-1}\cle 
\e^{\nu/2}\left(\e^{(2-\nu)/4}\right)^{j-1}$$
independently of $j$. The sum of the series 
$\sum (j+1)\,a_{j+1} \left(\e^{(2-\nu)/4}\right)^{j-1}$ is uniformly bounded, and thus our proof is complete.  
\end{proof}

We now prove that the map \eqref{LICHN:contr} is a contraction of a small ball in $C^{k,\alpha,\nu-1}(M_\e)$ to itself. 

\begin{prop}
Let  $\alpha\in (0,1)$ and $\nu\in\left(\tfrac{3}{2},2\right)$. For sufficiently large $c$ and sufficiently small $\e$,
the map $\mathcal{P}_\e$
is a contraction of the closed ball
of radius $c\,\e^{\nu/2}$ around $0$ in $C^{k,\alpha,\nu-1}(M_\e)$.
\end{prop}

\begin{proof}
Let $\|\eta\|_{k,\alpha,\nu-1}\le c\,\e^{\nu/2}$ for some $c$ which is determined below. Proposition \ref{LICHN:Qest} and the fact that  $\Q_\e(0)=0$ show that 
$\|\Q_\e(\eta)\|_{k,\alpha,\nu+1}\le c''\e^\nu$
for some $c''$. 
We now use Proposition \ref{NE:est} to see that, for sufficiently small $\e$, there is a constant $\tilde{c}$ independent of $c$ and $c''$ such that 
$$\|\Ne(\phi_0)+\Q_\e(\eta)\|_{k,\alpha,\nu+1}\le \tilde{c}\,\e^{\nu/2}.$$
This means  that   for sufficiently small $\e$ we have 
$$\|\mathcal{P}_\e(\eta)\|_{k,\alpha,\nu-1}=\big\|(\L_\e)^{-1}\big(\Ne(\phi_0)+\Q_\e(\eta)\big)\big\|_{k,\alpha,\nu-1}\le \tilde{c}\,l\e^{\nu/2},$$
where $l$ is the (uniform) upper bound on the inverses of linearized Lichnerowicz operators discussed in Theorem \ref{LICHN:full}. 
By choosing $c\ge\tilde{c}\,l$  we ensure that 
$$\mathcal{P}_\e:\bar{B}_{c\,\e^{\nu/2}}\to \bar{B}_{c\,\e^{\nu/2}}.$$
To examine if this restriction of $\mathcal{P}_\e$ is a contraction let 
$\eta_1, \eta_2\in \bar{B}_{c\,\e^{\nu/2}}$, assume $\e$ is small and compute
$$\|\mathcal{P}_\e(\eta_1)-\mathcal{P}_\e(\eta_2)\|_{k,\alpha,\nu-1}=\big\|(\L_\e)^{-1}\big(\Q_\e(\eta_1)-\Q_\e(\eta_2)\big)\big\|_{k,\alpha,\nu-1} \le c'l\,\e^{\nu/2}\|\eta_1-\eta_2\|_{k,\alpha,\nu-1};$$
the constant $c'$ used here is the one discussed in Proposition \ref{LICHN:Qest}.
The last estimate implies that, for $\e$ small enough, the map $\mathcal{P}_\e:\bar{B}_{c\,\e^{\nu/2}}\to \bar{B}_{c\,\e^{\nu/2}}$ is a contraction.
\end{proof}

Applying the Banach Fixed Point Theorem to 
$\mathcal{P}_\e:\bar{B}_{c\,\e^{\nu/2}}\to \bar{B}_{c\,\e^{\nu/2}}$ we obtain our  main result regarding the Lichnerowicz equation \eqref{LICHN}.

\begin{prop}
Let  $\alpha\in(0,1)$ and $\nu\in\left(\tfrac{3}{2},2\right)$. If $\e$ is sufficiently small, there exists a function $\phi_\e$ on $M_\e$ which solves the Lichnerowicz equation
$$\Delta_{g_\e}\phi_\e-\tfrac{1}{8}R(g_\e)\phi_\e+\tfrac{1}{8}|\me|_{g_\e}^2\phi_\e^{-7}-\left(\tfrac{\Lambda}{4}-\tfrac{\t^2}{12}\right)\phi_\e^5=0.$$
The function $\phi_\e$ is a small perturbation of the constant function $\phi_0\equiv1$; that is 
\begin{equation}\label{phi:est}
\|\phi_\e - \phi_0\|_{k,\alpha,\nu-1}\cle \e^{\nu/2}.
\end{equation}
\end{prop}

The outcome of our work is (the family of) data 
$\left(\phi_\e^4g_\e, \phi_\e^{-2}\me+\tfrac{\t}{3}\phi_\e^4g_\e\right)$ on $M_\e$ which by construction satisfy the momentum and the Hamiltonian constraints with the cosmological constant $\Lambda$. What remains to be discussed is whether these initial data obey the point-particle limit properties discussed in Theorem \ref{MAINTHM}. 

First consider a compact subset $\mathbf{K}\subseteq M\smallsetminus\{S\}$; let  $i_\e:\mathbf{K}\to M_\e$ 
be the corresponding embedding. 
Since $1\cle w_\e$ on $i_\e(\mathbf{K})$ the estimate \eqref{phi:est} yields 
$$\left\|\left(\phi_\e^4\circ i_\e\right)-1\right\|_{C^k(\mathbf{K},g)}\cle \e^{\nu/2}.$$
It now follows from the definition of $g_\e$ that 
\begin{equation}\label{ordinarylim:g}
\left\|(i_\e)^*\left(\phi_\e^4g_\e\right) -g\right\|_{C^k(\mathbf{K},g)}=O(\e^{\nu/2}) \text{\ \ as\ \ }\e\to 0.
\end{equation}
Likewise, the estimate 
$\left\|\left(\phi_\e^{-2}\circ i_\e\right)-1\right\|_{C^k(\mathbf{K},g)}\cle \e^{\nu/2}$
together with \eqref{ordinarylim:mu} and \eqref{ordinarylim:g} implies 
$$
\left\|(i_\e)^*\left(\phi_\e^{-2}\me+\tfrac{\t}{3}\phi_\e^4g_\e\right) -K\right\|_{C^k(\mathbf{K},g)}=O(\e^{\nu/2}) \text{\ \ as\ \ }\e\to 0.
$$
On the other hand, consider a compact set $\mathbf{K}\subseteq M_0$ and the corresponding embedding  $\iota_\e:\mathbf{K}\to M_\e$. Note that $w_\e \sim \e$ on $\iota_\e(\mathbf{K})$ and that, therefore, for each given $j$ with $0\le j\le k$ we have  
$$\e^{j+(\nu-1)}\left|\genabla^j\left(\left(\phi_\e^4\circ \iota_\e\right)-1\right)\right|_{g_\e}=\e^{\nu-1}\left|\gonabla^j\left(\left(\phi_\e^4\circ \iota_\e\right)-1\right)\right|_{g_0}\cle \e^{\nu/2}.$$
We now see that 
$$\left\|\left(\phi_\e^4\circ \iota_\e\right)-1\right\|_{C^k(\mathbf{K},g)}\cle \e^{1-\nu/2}
\text{\ \ and\ similarly\ \ }
\left\|\left(\phi_\e^{-2}\circ \iota_\e\right)-1\right\|_{C^k(\mathbf{K},g)}\cle \e^{1-\nu/2}.
$$
Using the definition of $g_\e$ one readily verifies 
$$\left\|\tfrac{1}{\e^2}(\iota_\e)^*\left(\phi_\e^4g_\e\right)-g_0\right\|_{C^k(\mathbf{K},g)}=O \left(\e^{1-\nu/2}\right) \text{\ \ as\ \ }\e\to 0.$$
Combining this with \eqref{scaledlim:mu} further gives
$$\left\|\tfrac{1}{\e}\left(\iota_\e\right)^*
\left(\phi_\e^{-2}\me+\tfrac{\t}{3}\phi_\e^4g_\e\right)-K_0\right\|_{C^k(\mathbf{K},g_0)}=O\left(\e^{1-\nu/2}\right) \text{\ \ as\ \ }\e\to 0.$$
Thus, the initial data $\left(\phi_\e^4g_\e, \phi^{-2}\me+\tfrac{\t}{3}\phi_\e^4g_\e\right)$ obey the point-particle limit properties discussed in the statement of Theorem \ref{MAINTHM}.

\appendix
\section{On the {\sc CKVF} and {\sc Injectivity assumptions}} V. Moncrief pointed to us that these assumptions are related to non-existence of Killing vector fields in a space-time development  of $(M,g,K)$. Subsequently, we examined if the computations of \cite{VM} can be modified to accommodate for the presence of a cosmological constant and concluded that if $M$ is orientable and the initial data $(M,g,K)$ satisfy {\sc CKVF} and {\sc Injectivity assumptions} then no space-time development  of $(M,g,K)$ has a Killing vector field. We outline the argument below. 

Suppose there exists a Killing vector field $\mathbf{X}$ on a  space-time development of $(M,g,K)$. Let $\mathbf{n}$ denote a unit normal vector field to the Cauchy slice $(M,g,K)$ within this space-time development. We may write $\mathbf{X}=C\mathbf{n}+X$ where $X$ is tangent to $M$ and $C$ is some function on $M$. Note that the {\sc CKVF assumption} implies $C\not\equiv 0$. A lengthy computation which follows \cite{VM} shows that 
\begin{equation}\label{KIDs}
\begin{cases}
0=-2CK+\L_Xg&\ \\
0=-\mathrm{Hess}_g(C)+\left(\mathrm{Ricci}_g-2K^2+\t K-\Lambda g\right)C+\L_X K,&\ 
\end{cases}
\end{equation}
where $K^2_{ab}:=K_a^cK_{cb}$ and where $\L_X$ denotes the Lie derivative. Tracing the second equation, using the first equation to eliminate the Lie derivative term, and using the Hamiltonian constraint to eliminate the scalar curvature term yields
$$0=-\Delta_g C+\left(|K|^2_g-\Lambda\right)C.$$
By {\sc Injectivity assumption} we must have $C\equiv 0$. This produces a contradiction.   

Readers familiar with the notion of Killing Initial Data have surely noticed that \eqref{KIDs} are the defining equations for KIDs on $(M,g,K)$ (see  \cite{KIDS}) and  that our assumptions imply $(M,g,K)$ has no KIDs.

\section{The Euclidean Vector Laplacian in Spherical Coordinates}
The Euclidean vector Laplacian $L_\delta$ can be understood explicitly using spherical harmonics. Let $r:x\mapsto |x|$ and let $\left\{\phi_{l,m}\ \big|\ l\in \N\cup\{0\}, m\in \mathbb{Z}\cap [-l, l]\right\}\subseteq L^2(S^2)$ be the complete orthonormal basis of eigenfunctions for the Hodge Laplacian $\Delta_H$ on $S^2$:   
$$\Delta_H\phi_{l,m}=\delta d \phi_{l,m}=l(l+1) \phi_{l,m}.$$
We note that $\phi_{0,0}=\frac{1}{2\sqrt{\pi}}$ is a constant function. 
The set of $1$-forms 
$$\left\{\frac{d\phi_{l,m}}{\sqrt{l(l+1)}}\ \Big|\ l\in \N,\ \ m\in \mathbb{Z}\cap [-l,l]\right\}\cup \left\{ \frac{\delta\star\phi_{l,m}}{\sqrt{l(l+1)}}\ \Big|\ l\in \N,\ \ m\in \mathbb{Z}\cap [-l,l]\right\}$$ is a complete orthonormal basis for the $L^2$-space of $1$-forms on $S^2$. 
For notational 
simplicity we set 
$$\lambda_l:=l(l+1),\ \ \ \mathbf{V}_{l,m}:=\left(\frac{d\phi_{l,m}}{\sqrt{\lambda_l}}\right)^{\sharp} \text{\ and\ } \ \mathbf{W}_{l,m}:=\left(\frac{\delta\star\phi_{l,m}}{\sqrt{\lambda_l}}\right)^{\sharp},$$
where $^{\sharp}$ denotes the vector field dual with respect to the standard round metric on $S^2$. In what follows we abuse  notation and let  $\mathbf{V}_{l,m}$, $\mathbf{W}_{l,m}$ denote the trivial ($r$-independent) extensions of the vector fields to $\R^3\smallsetminus\{0\}$. We point out that $\mathbf{V}_{l,m}$ and $\mathbf{W}_{l,m}$ are only defined for $l\neq 0$. 

A vector field $Y$ on $\R^3\smallsetminus\{0\}\approx (0,+\infty)\times S^2$ can be decomposed as 
$$Y=f_r\cdot \partial_r+Z_r,$$
where $f_r:S^2\to \R$ is a one-parameter family of functions and where $Z_r$ is a one-parameter family of vector fields on $S^2$. Using spherical harmonics we may decompose 
$$f_r=\sum_{L^2(S^2)} u_{l,m}(r)\phi_{l,m}, \ \ \ Z_r=\sum_{L^2(S^2)} \left[v_{l,m}(r)\mathbf{V}_{l,m}+w_{l,m}(r)\mathbf{W}_{l,m}\right].$$
A lengthy computation involving Weitzenb\o ck formulae 
(see \cite[\S 3.1]{IMP}) leads to:
\begin{equation*}
\begin{aligned}
L_\delta Y=&-\left(\frac{2}{3}u_{0,0}''+\frac{4}{3}\frac{u_{0,0}'}{r}-\frac{4}{3}\frac{u_{0,0}}{r^2}\right)\phi_{0,0}\partial_r\\
-&\sum_{L^2(S^2)}\left[\frac{2}{3}u_{l,m}''+\frac{4}{3}\frac{u_{l,m}'}{r}-\frac{4}{3}\frac{u_{l,m}}{r^2}-\frac{\lambda_l}{2}\frac{u_{l,m}}{r^2}+\frac{\sqrt{\lambda_l}}{r}v_{l,m}-\frac{\sqrt{\lambda_l}}{6}v_{l,m}'\right]\phi_{l,m}\partial_r\\
-&\sum_{L^2(S^2)} \left[\frac{\sqrt{\lambda_l}}{6}u_{l,m}'+\frac{4\sqrt{\lambda_l}}{3}\frac{u_{l,m}}{r}+2rv_{l,m}'+\left(1-\frac{2\lambda_l}{3}\right)v_{l,m}+\frac{r^2}{2} v_{l,m}''\right]\mathbf{V}_{l,m}\\
-&\sum_{L^2(S^2)}\left[\frac{r^2}{2}w_{l,m}''+2rw_{l,m}'+\left(1-\frac{\lambda_l}{2}\right)w_{l,m}\right] \mathbf{W}_{l,m}.
\end{aligned}
\end{equation*}
This decoupling allows us to find the kernel of $L_\delta$ \emph{explicitly}: it is spanned by 
\begin{description}
\item[{\sc $\bullet$ Basis vector fields which blow up at $\infty$}] The vector field $r\partial_r$ and three families indexed by $l\in \mathbb{N}$ and $m\in\mathbb{Z}\cap [-l,l]$:
\begin{gather*}
\{(l-6)\sqrt{\lambda_l}\cdot r^{l+1}\phi_{l,m}\partial_r+l(l+9)r^l\mathbf{V}_{l,m}\},\\ 
\{\sqrt{\lambda_l}\cdot r^{l-1}\phi_{l,m}\partial_r+(l+1)r^{l-2}\mathbf{V}_{l,m}\}\ \ 
\text{and} \ \ \{r^{l-1}\mathbf{W}_{l,m}\}. 
\end{gather*}
\medbreak
\item[{\sc $\bullet$ Basis vector fields which blow up at the origin}] The vector field $r^{-2}\partial_r$ and three families indexed by $l\in \mathbb{N}$ and $m\in\mathbb{Z}\cap [-l,l]$: 
\begin{gather*}
 \left\{(l+7)\sqrt{\lambda_l}\cdot r^{-l}\phi_{l,m}\partial_r - (l+1)(l-8)r^{-l-1}\mathbf{V}_{l,m}\right\},\\ \left\{\sqrt{\lambda_l}\cdot r^{-l-2}\phi_{l,m}\partial_r - l r^{-l-3}\mathbf{V}_{l,m}\right\}\ 
\text{and}\ \{r^{-l-2}\mathbf{W}_{l,m}\}.
\end{gather*}
 
\end{description}
An alternative way of obtaining the contradiction in {\sc Case $\R^3\smallsetminus\{0\}$} of Proposition \ref{VL:MAIN} is to show 
$$C^{0,0,\nu}_{1,0}(\R^3;0,\infty)\cap \mathrm{ker}(L_\delta)=\{0\}$$
for $\nu\in(1,2)$.  
This identity is easy to prove from the above description of $\mathrm{ker}(L_\delta)$.

\end{document}